\documentclass[11pt,a4paper,reqno]{amsart}
\usepackage [latin9]{inputenc}

\usepackage{a4wide}
\usepackage{amssymb,amsmath,amsthm,mathrsfs}
\usepackage{graphicx}
\usepackage{float}
\usepackage{colortbl}
\usepackage{enumerate}
\usepackage{upref}
\usepackage{stackrel}
\usepackage{enumitem}
\usepackage[bookmarks=false]{hyperref}
\usepackage[T2A,T1]{fontenc}
\DeclareSymbolFont{cyrillic}{T2A}{cmr}{m}{n}
\DeclareMathSymbol{\D}{\mathalpha}{cyrillic}{196}

\usepackage{accents}
\newlength{\dhatheight}

\theoremstyle{plain}
\newtheorem{theorem}{Theorem}[section]

\newtheorem{corollary}[theorem]{Corollary}

\theoremstyle{definition}
\newtheorem{definition}{Definition}[section]

\newtheorem{example}{Example}[section]
\newtheorem*{condition}{Condition}

\theoremstyle{remark}
\newtheorem{remark}[theorem]{Remark}

\usepackage{enumitem}
\makeatletter
\def\namedlabel#1#2{\begingroup
   #2%
 \def\@currentlabel{#2}%
   \phantomsection\label{#1}\endgroup
}

\def\R{\ensuremath{\mathbb R}}
\def\N{\ensuremath{\mathbb N}}

\def\I{\ensuremath{{\bf 1}}}
\def\e{{\ensuremath{\rm e}}}

\def\p{\ensuremath{\mathbb P}}

\def\X{\mathcal{X}}

\def\1{{\bf 1}}

\def\dist{\ensuremath{\text{dist}}}

\def\ie{{\em i.e.}, }

\def\E{\mathbb E}

\def\dist{\ensuremath{\text{dist}}}

\def\shift{\mathcal{T}}

\numberwithin{equation}{section}

\makeatletter
\def\moverlay{\mathpalette\mov@rlay}
\def\mov@rlay#1#2{\leavevmode\vtop{%
   \baselineskip\z@skip \lineskiplimit-\maxdimen
   \ialign{\hfil$\m@th#1##$\hfil\cr#2\crcr}}}
\newcommand{\charfusion}[3][\mathord]{
    #1{\ifx#1\mathop\vphantom{#2}\fi
        \mathpalette\mov@rlay{#2\cr#3}
      }
    \ifx#1\mathop\expandafter\displaylimits\fi}
\makeatother

\newcommand{\cupdot}{\charfusion[\mathbin]{\cup}{\cdot}}
\newcommand{\bigcupdot}{\charfusion[\mathop]{\bigcup}{\cdot}}

\begin{document}

\title[]{Dynamical counterexamples regarding the Extremal Index and the mean of the limiting cluster size distribution}

\author[M. Abadi]{Miguel Abadi}
\address{Miguel Abadi\\ Instituto de Matem\'atica e Estat\'istica\\Universidade de S. Paulo\\ Rua do Mat\~ao 1010\\ Cid. Universitaria\\ 05508090 - São Paulo\\ SP - Brasil} \email{leugim@ime.usp.br}
\urladdr{\url{http://miguelabadi.wixsite.com/miguel-abadi}}

\author[A. C. M. Freitas]{Ana Cristina Moreira Freitas}
\address{Ana Cristina Moreira Freitas\\ Centro de Matem\'{a}tica \&
Faculdade de Economia da Universidade do Porto\\ Rua Dr. Roberto Frias \\
4200-464 Porto\\ Portugal} \email{\href{mailto:amoreira@fep.up.pt}{amoreira@fep.up.pt}}
\urladdr{\url{http://www.fep.up.pt/docentes/amoreira/}}

\author[J. M. Freitas]{Jorge Milhazes Freitas}
\address{Jorge Milhazes Freitas\\ Centro de Matem\'{a}tica \& Faculdade de Ci\^encias da Universidade do Porto\\ Rua do
Campo Alegre 687\\ 4169-007 Porto\\ Portugal}
\email{\href{mailto:jmfreita@fc.up.pt}{jmfreita@fc.up.pt}}
\urladdr{\url{http://www.fc.up.pt/pessoas/jmfreita/}}

\thanks{All authors were partially supported by the joint project FAPESP (SP-Brazil) and FCT (Portugal) with reference FAPESP/19805/2014. ACMF and JMF were partially supported by FCT projects  PTDC/MAT-CAL/3884/2014 and PTDC/MAT-PUR/28177/2017, with national funds, and by CMUP (UID/MAT/00144/2013), which is funded by FCT with national (MCTES) and European structural funds through the programs FEDER, under the partnership agreement PT2020.}

\date{\today}



\begin{abstract}
The Extremal Index is a parameter that measures the intensity of clustering of rare events and is usually equal to the reciprocal of the mean of the limiting cluster size distribution. We show how to build dynamically generated stochastic processes with an Extremal Index for which that equality does not hold. The mechanism used to build such counterexamples is based on  considering observable functions maximised at at least two points of the phase space, where one of them is an indifferent periodic point and another one is either a repelling periodic point or a non periodic point. The occurrence of extreme events is then tied to the entrance and recurrence to the vicinities of those points. This enables to mix the behaviour of an Extremal Index equal to $0$ with that of an Extremal Index larger than $0$. Using bi-dimensional point processes we explain how mass escapes in order to destroy the usual relation. We also perform a study about the formulae to compute the cluster size distribution introduced earlier and prove that ergodicity is enough to establish that the finite versions of the reciprocal of the Extremal Index and of the mean of the cluster size distribution do coincide.
\end{abstract}

\maketitle
\tableofcontents

\section{Introduction}

In Extreme Value Theory, the study of rare events is tied with the observation of abnormally high values among a series of realisations of a certain variable of interest, $X_0, X_1, \ldots$, \ie one is interested in events of the form $\{X_i>u\}$, corresponding to the exceedance of a high threshold $u$. The choice of the level $u$ is usually made according to the size $n$ of the available sample so that the average number of exceedances  is asymptotically a constant, $\tau\geq0$, as $n\to\infty$ (see \eqref{un}). A classical problem in this setting is the study of the asymptotic distribution of the partial maximum $M_n:=\max\{X_0, \ldots, X_{n-1}\}$.      

For dependent data, the exceedances may present a tendency to cluster, \ie to appear in groups rather than scattered along the time line. There is then an important parameter that quantifies the intensity of clustering of extreme events. This parameter, which we will denote by $\theta$, was called Extremal Index (EI) by Leadbetter in \cite{L83}, and is usually defined from the asymptotic distribution of $M_n$, namely, from the limit: $\lim_{n\to\infty}\p(M_n\leq u_n)=\e^{-\theta\tau}$. (See Definition~\ref{def:EI}). The EI takes values in $[0,1]$ and is such that $\theta=1$ means absence of clustering while $\theta$ close to $0$ means intensive clustering. 

In order to keep track of the occurrence of extreme events, one can consider point processes that count the number of exceedances on a normalised time frame. Although, we defer to Section~\ref{subsec:point-processes} its formal definition, we advance that, for a given size sample, $n$, and up to a given instant, $t$, the point processes give us the quantity:
$$
N_n(t)=\sum_{i=0}^{\lceil nt\rceil-1}\I_{\{X_i>u_n\}}.
$$
  In \cite{HHL88}, these point processes were proved to converge to a compound Poisson process where the Poisson events are charged by a multiplicity corresponding to the cluster size, \ie in particular, $N_n(t)$ was proved to converge in distribution to $N(t)=\sum_{i=1}^{N^*(t)}D_i$, where $N^*(t)$ is a Poisson random variable of mean $\theta\tau t$ independent of $D_1, D_2,\ldots$, which is an independent and identically distributed sequence of positive integer valued random variables. (See formal definition in Section~\ref{subsec:point-processes}).  Under some regularity conditions, the EI can be identified as the inverse of the mean cluster size, \ie $\theta^{-1}=\E(D_1)$ is the average of the multiplicity distribution of the limiting compound Poisson process. When $\theta=1$, the cluster size is $1$ a.s. ($D_1=1$ a.s.) and the limiting process is a simple Poisson process.

However, in \cite{S88} a counterexample was given where the EI does not coincide with the inverse of the mean cluster size of the limiting compound Poisson process, \ie $\theta^{-1}\neq\E(D_1)$. This example is based on a regenerative sequence with EI equal to $1/2$ but with a simple Poisson process limit, which means a mean cluster size equal to $1$. The regenerative property of the sequence is the key to prove the existence of an EI equal to $1/2$, which is guaranteed by \cite[Theorem~3.1]{R88}.

More recently, a theory of extreme values for dynamical systems has been developed (see \cite{LFFF16} and references therein). The idea is to consider stochastic processes arising from dynamical systems by evaluating a given observable along the orbits of that system. Namely, letting $T:\mathcal X\to\mathcal X$ be a discrete time dynamical system and $\varphi:\mathcal X\to \R$ a measurable function defined on the probability space $(\X,\mathcal B_{\mathcal X}, \mu)$, (where $\mu$ is $T$-invariant), then we define $X_0, X_1,,\ldots$ by $X_n=\varphi\circ T^n$, where $T^n$ denotes the $n$-fold composition of $T$ with itself. This observable $\varphi$ is typically maximised at a single point $\zeta$ chosen in the phase space $\mathcal X$ and then, as observed in \cite{FFT10}, the study of the occurrence of extreme events is related to problems of entrance and recurrence times. In \cite{FFT12}, the authors have shown that periodicity of $\zeta$ implies the appearance of clustering and, consequently, an EI less than 1, which is given by the rate of expansion of the system at the maximal point $\zeta$. Later, in \cite{FFT13}, the authors showed that at periodic points the Rare Events Point Process (REPP) converge to a compound Poisson process with a geometric multiplicity distribution of average $\theta^{-1}$, \ie $\p(D_1=j)=\theta(1-\theta)^{j-1}$. Moreover, for sufficiently regular systems, a full dichotomy exists (see \cite{K12,AFV15}), \ie either $\zeta$ is periodic and we have clustering or $\zeta$ is non-periodic and we have the absence of clustering with $\theta=1$ and a standard Poisson process as a limit for the REPP. In \cite{AFFR16}, the authors introduced a new device to create clustering: instead of considering observables maximised at a single point, they consider multiple maximising points and show that if these points are related by belonging to the same orbit then a fake periodic behaviour emerges, which is responsible for the appearance of clustering of extreme observations. In this case, the maximal points need not to be periodic but the maximal set that they form, \ie the set of points where the observable attains the global maximum of the observable $\varphi$, is periodic in the sense that it recurs to itself after a finite number of iterations. This approach yielded examples of different clustering patterns corresponding to different multiplicity distributions pertaining to the cluster size. However, in all such examples the EI coincides with inverse of the mean cluster size.

In this paper, we use the same mechanism to produce new counterexamples of stochastic processes with an EI that cannot be interpreted as the inverse of the mean cluster size of the corresponding limiting process. The idea is to consider an observable maximised at (at least) two points, where one of them is an indifferent periodic point while the other is either a non-periodic point or a repelling periodic point. We recall that when an observable is maximised at a single indifferent fixed point, we obtain an EI $\theta=0$ (see \cite{FFTV16}). Hence, we are mixing a degenerate behaviour corresponding to an EI equal to 0 with an EI strictly larger than 0 to obtain a stochastic process with an EI, which somehow corresponds to an average of these two types of behaviour, but whose finite time multiplicity distributions are not uniformly integrable and, therefore, the mean of the respective limit does not coincide with the inverse of the EI. To prove these statements we will use the formulas for EI given in \cite{FFT12}, the formulas for the multiplicity distributions given in \cite{FFT13, AFV15}, the dynamics of the Manneville-Pomeau map and also some tools from \cite{FFTV16}. 

We remark that, in the counterexample built by Smith in \cite{S88}, one can show that the regenerative process is also combining the behaviour of an EI equal to $0$ and an EI equal to $1$, which we defer to \cite{AFF18}. Moreover, in all counterexamples, the EI still coincides with the reciprocal of the limit of the means of the finite time cluster size distributions. Hence, the problem is that the limit of the mean finite time cluster size distribution does not coincide with the mean of the limiting cluster size distribution. This happens because there exists an escape of mass, which can be detected by looking at bi-dimensional point processes of rare events, which can be projected to obtain the one dimensional REPP mentioned earlier. In Section~\ref{sec:escape-of-mass} , we describe how the behaviour corresponding to an EI equal to $0$ is responsible for the escape of mass observed in the counterexamples, which ultimately explains why the usual interpretation for the EI fails in these situations.

Another highlight of this paper is the fact that we provide a nice interpretation of the formula to compute the cluster size distribution of the limiting process that was introduced in \cite{FFT13,AFV15} and relate it to the one used by Robert in \cite{R13}, for example. Moreover, we prove that ergodicity is sufficient to show that the EI still coincides with reciprocal of the limit of the mean finite time cluster size distribution.

\section{Extremal analysis of stationary stochastic processes}

In this section we let $X_0, X_1, \ldots$ denote a general stationary stochastic process, which we identify with the respective coordinate-variable process on $(\mathcal R^{\N_0}, \mathcal B^{\N_0}, \p)$, where $\mathcal R=\R^d$ and  $\mathcal B^{\N_0}$ is the $\sigma$-field generated by the coordinate functions $V_n:\mathcal R^{\N_0}\to\mathcal R$, with $V_n(x_0,x_1,\ldots)=x_n$, for $n\in\N_0$, so that there is a natural measurable map, the shift operator $\shift: \mathcal R^{\N_0}\to\mathcal R^{\N_0}$, given by $\shift(x_0,x_1,\ldots)=(x_1,x_2,\ldots)$, which when applied later in the dynamical systems context can be identified with $T$.

Observe that:
\[
V_{i-1}\circ \shift =V_{i}, \quad \mbox{for all $i\in\N$}.
\]
Since, we are assuming that the process is stationary, then $\p$ is $\shift$-invariant. Note that $V_i=V_0\circ \shift^i$, for all $i\in\N_0$, where $\shift^i$ denotes the $i$-fold composition of $\shift$, with the convention that $\shift^0$ denotes the identity map on $\mathcal R^{\N_0}$. 

In what follows, for every $A\in\mathcal B$, we denote the complement of $A$ as $A^c:=\mathcal X\setminus A$.

\subsection{Clustering of rare events}
\label{subsec:clustering}

Consider an extreme or rare event $A\in\mathcal B$ whose occurrence we want to study. For independent and identically distributed (iid) stochastic processes we expect the occurrences of $A$ to appear scattered along the time line. When the random variables are not independent then there may be a tendency for the observations of $A$ to appear concentrated in groups (clusters). This is sometimes referred as the law of series, see \cite{DL11}. Identifying the clusters becomes a problem because sometimes is not clear if a certain observation of $A$ is sufficiently close (in time) to others in order to be classified as belonging to the same cluster. 

There are two main methods to identify clusters. One is called the \emph{block declustering scheme} and the other one is the \emph{runs declustering scheme} (see \cite{SW94,F03a}). The block method splits the $n$ observations into $k_n$ (or $k_n+1$) blocks of length $\lfloor n/kn\rfloor$ and then establishes that any extreme events within the same block belong to the same cluster. The runs declustering scheme consists in setting a run length, $q$, and establishing that any rare events separated by at most $q-1$ non-extreme observations must belong to the same cluster, so that between two distinct clusters there must be a run of at least $q$ non-extreme observations.   

We are going to assume a runs declustering scheme and, therefore, we consider that there exists a fixed $q\in \N$ which will be the maximum waiting time between the  occurrence of two extreme events on the same cluster. In the applications, in order to prove the convergence of the point processes we will consider a condition on the dependence structure of the stochastic process called $\D'_q(u_n)^*$ (see Section~\ref{subsec:convergence-PP}), which when satisfied implies that the two declustering approaches behave essentially in the same way (see Remark~\ref{rem:D'-declustering-relation}).  

The choice of the run length $q$ is quite sensitive and we will return to the subject in Section~\ref{subsec:convergence-PP} when we introduce and discuss condition $\D'_q(u_n)^*$. However, we advance here the following interpretation for the value $q$. In \cite{FFT12}, we introduced the EI in the dynamical setting and established a relation between the appearance of clustering and the existence of  underlying periodic phenomena in the structure of the stochastic process. In fact, in the dynamical context, as observed, in \cite{FFT12} and in the subsequent papers \cite{AFFR16,AFFR17,FFRS19}, clustering is directly related with the periodicity of the maximal set $\mathcal M$, \ie the set of points where the observable $\varphi$ achieves the global maximum. Hence, $q$ can be interpreted as the largest of the periods of the underlying periodic phenomena present in the stochastic process. 

In order to illustrate the appearance of clustering and suitable choices for $q$ we give some simple examples of stationary stochastic processes, some of them arising from dynamical systems.

\begin{example}{A Maximum Moving Average process with period $2$.}
\label{exp:MMA-101}

Let $Y_{-2},Y_{-1},Y_0,Y_1,\ldots$ be a sequence of iid random variables with common continuous distribution function $G$. We define a Maximum Moving Average process $X_0,X_1,\ldots$ based on the previous sequence in the following way: for each $n\in\N_0$ set
\begin{equation*}
X_n=\max\{Y_{n-2},Y_n\}.
\end{equation*}
Without going into too much detail (for which we refer to  \cite[Appendix~A]{FFT12}), in this case it is clear that we have an underlying periodic phenomenon of period 2. To see this, observe that if $X_0>u$, for some large $u$, then there is a very good chance that $X_2>u$, as well. In fact, in this case, clusters will be separated by at least 2 observations below the reference threshold and the cluster size is 2 a.s. (see end of Section~\ref{subsec:convergence-PP}).

\end{example}

\begin{example}{A dynamically generated process based on a periodic repelling point.}
\label{exp:geometric}

Consider the discrete time dynamical system $T:[0,1]\to[0,1]$ given by $T(x)=2x\mod1$, equipped with Lebesgue measure on the borelean subsets of $[0,1]$ and let $\varphi:[0,1]\to\R\cup\{+\infty\}$ be such that $\varphi(x)=-\log|x-1/3|$. Define $X_0, X_1, \ldots$ by $X_n=\varphi\circ T^n$. This particular example was studied in \cite[Example~4.2.1]{LFFF16}. Observe that $\{X_0>u\}=(1/3-\e^{-u},1/3+\e^{-u})$ and $\zeta=1/3$ is a repelling periodic point of period 2, since $T(\zeta)=2/3$ and $T^2(\zeta)=1/3=\zeta$. This means that if  our orbit starts very close to $1/3$, we will observe a sequence of exceedances at even time steps until eventually the orbit leaves the set $(1/3-\e^{-u},1/3+\e^{-u})$ and then, typically, we have to wait a long time (because we are assuming that $u$ is large) until the systems brings it back again to that set in order to observe a new cluster of exceedances. In this case, it is also clear that one should take $q=2$.
\end{example}

\begin{example}{A dynamically generated process based on two periodic repelling points.}
\label{exp:mixture}

Consider the same discrete time dynamical system of the previous example and let $\varphi:[0,1]\to\R\cup\{+\infty\}$ be such that $\varphi(x)=-\log|x-1/3|\I_{[0,1/2]}-\log|x-5/7|\I_{(1/2,1]}$. Again, define $X_0, X_1, \ldots$ by $X_n=\varphi\circ T^n$. Observe that $\{X_0>u\}=(1/3-\e^{-u},1/3+\e^{-u})\cup (5/7-\e^{-u},5/7+\e^{-u})$ and $\zeta_1=1/3$ is a repelling periodic point of period 2, while $\zeta_2=5/7$ is a repelling periodic point of period 3. Hence, if we start with an exceedance, then either we observe a cluster of excedances observed at even time steps (if one starts very close to $\zeta_1$) or a cluster of exceedances observed at time steps which are multiples of 3 (if one starts very close to $\zeta_2$). In this case, in order to be sure that a cluster has ended, one has to observe a run of at least 3 non-exceedances, which means that $q=3$.
\end{example}

Having fixed a run length $q\in\N$, we define the sequence of nested sets $\left(U^{(\kappa)}(A)\right)_{\kappa\geq0}$ of $\mathcal B^{\N_0}$ given by:
\begin{align*}
U^{(0)}(A)&
=V_0^{-1}(A)\\ Q_q^{(0)}(A)&
=U^{(0)}(A)\cap \bigcap_{i=1}^q \shift^{-i}((U^{(0)}(A))^c),
\end{align*}
and for $\kappa\in\N$,
\begin{align}
U^{(\kappa)}(A)&=U^{(\kappa-1)}(A)\setminus Q_q^{(\kappa-1)}(A)\label{def:Uk}
\\
Q_q^{(\kappa)}(A)&:=U^{(\kappa)}(A)\cap\bigcap_{i=1}^q \shift^{-i}\left((U^{(\kappa)}(A))^c\right)\label{def:Qk}
\\
U^{(\infty)}(A)&=\bigcap_{\kappa\geq 0} U^{(\kappa)}(A).\label{def:Uinfty}
\end{align}

Note that $U^{(\kappa-1)}(A)$ corresponds to observing $A$ at time $0$ and then observing $A$ for at least $\kappa$ times so that the waiting time between two observations of $A$ is at most $q$. 
The event $Q_q^{(\kappa-1)}(A)=U^{(\kappa-1)}(A)\setminus U^{(\kappa)}(A)$ corresponds to observing $A$ exactly $\kappa$ times within no more than $q$ units of time between one and the next observation of $A$. This means, in particular, that the  $\kappa+1$-th observation of $A$ occurs at least $q+1$ iterations after the  $\kappa$-th observation of $A$. 
The event $U^{(\infty)}(A)$ corresponds to the occurrence of an observation of $A$, which is followed by an infinite number of observations of $A$ which are at most $q$ units of time apart from each other. To put it in a different way, if we define 
\begin{eqnarray}
h:&\mathcal R^{\N_0}&\rightarrow \{0,1\}^{\N_0}\\
&\underline x=(x_0, x_1,\ldots)&\mapsto  \underline \omega=\omega_0\omega_1\ldots \nonumber
\end{eqnarray} by setting for each $n\in\N_0$ that $\omega_n=V_n(h(\underline x))=1$ if $x_n\in A$ and  $\omega_n=V_n(h(\underline x))=0$ if $x_n\notin A$, then if $\underline x=(x_0,x_1,\ldots)\in U^{(\infty)}(A)$ then $h(\underline x)$ is a binary sequence, which starts with a 1 and has no block of more than $q-1$ consecutive 0's.
Let $J$ be an interval contained in $[0,\infty)$. We define
\begin{equation}
\label{eq:W-def}
\mathscr W_{J}(A):=\bigcap_{i\in J\cap \N_0}\shift^{-i}(V_0^{-1}(A^c)).
\end{equation}
Note that if $\underline x\in \mathscr W_{J}(A)$ means that $h(\underline x)$ has a block of consecutive 0's that correspond to the observations in $J\cap \N_0$.
We can now write a formula to determine the cluster sizer distribution of observations of $A$. We define the mass probability function $\pi_A$ supported on the positive integers by  
\begin{equation}
\label{def:pi-A}
\pi_A(\kappa)=\frac{\p(Q_{q}^{(\kappa-1)}(A))-\p(Q_{q}^{(\kappa)}(A))}{\p(Q_{q}^{(0)}(A))}, \quad \text{for each $\kappa\in\N$}.
\end{equation}
This formula for the finite time cluster size distribution was used first in \cite{FFT13} and explicitly written for the first time in \cite{AFV15}. It appeared subsequently in \cite{AFFR16,AFFR17}. This formula was derived during the proof of the convergence of REPP, which was based on a blocking type of argument. Although very useful it lacked a clear intuitive interpretation, which we mean to provide next.  

In order to establish the convergence of the REPP, we will describe a condition $\D'_q(u_n)^*$ inspired in condition $D'_p(u_n)^*$ from \cite{FFT13}, which is also very similar to the condition $D^{(k)}(u_n)$ introduced by Chernick et al. in \cite{CHM91}. This condition implies that the maximum waiting time before another observation of $A$ within the same cluster is $q$ units of time. Hence, if $(x_0,x_1,\ldots)\in\mathcal R^{\N_0}$ is a realisation of $X_0, X_1,\ldots$ then the beginning of cluster and the ending of cluster can be easily identified in $h(x_0,x_1,\ldots)$ by the appearance of a block of at least $q$ consecutive $0$'s.
Let $q,\kappa\in \N$ be fixed and consider the set of finite strings of 0's and 1's such that each string starts and ends with a 1, has exactly $\kappa$ 1's, which are separated by at most $q-1$ 0's, \ie there is no block of $q$ or more consecutive 0's in the string. Namely, let 
\begin{multline*}
W_q(\kappa)=\left\{\varpi\in \bigcup_{i=\kappa}^{q(\kappa-1)+1}\{0,1\}^i: V_0(\varpi)=V_{|\varpi|-1}(\varpi)=1,\, \sum_{i=0}^{|\varpi|-1}V_0(\shift^i(\varpi))=\kappa,\right.\\
 \left.\shift^i(\varpi)\in \bigcup_{j=0}^{q-1}V_j^{-1}(1),\, \mbox{for all $i=0,\ldots,|\varpi|-1$}\right\}, 
\end{multline*}
where we still use the notation $\shift$ and $V_j$ for the shift map and the projection on the $j$-th coordinate even when leading with finite strings and $|\varpi|$ is the length of the finite string $\varpi$. Finally we define:
$$
\mathcal H_q(\kappa)=h^{-1}\left(\left\{\underline \omega\in\{0,1\}^{\N_0}: \,\underline \omega=\underbrace{0\ldots 0}_{\mbox{ \tiny $q$ symbols}} \! \varpi \! \underbrace{0\ldots 0}_{\mbox{\tiny $q$ symbols}}\!\ldots,\;\mbox{for some $\varpi\in W_q(\kappa)$}\right\}\right)
$$
and also set
$$
\mathcal H_q(0)=h^{-1}\left(\left\{\underline \omega\in\{0,1\}^{\N_0}: \,\underline \omega=\underbrace{0\ldots 0}_{\mbox{ \tiny $q$ symbols}}\!\!\!1\ldots\right\}\right).
$$
Observe that $\mathcal H_q(0)$ determines the beginning of a new cluster, while $\mathcal H_q(\kappa)$ corresponds to the appearance of a cluster of size $\kappa$. Observe that $\mathcal H_q(\kappa)\subset\mathcal H_q(0)$ and to illustrate the definition we note that $0011010110100\ldots\in h(\mathcal H_2(6))$ and $0001001011000\ldots\in h(\mathcal H_3(4))$.

The next result gives an interpretation of $\pi_A$ defined in \eqref{def:pi-A} as the cluster size distribution, \ie as the probability of having a cluster of size $\kappa$ conditioned to knowing that we have initiated a cluster.
\begin{theorem}
Given $q\in\N$, consider the distribution $\pi_A$ given by \eqref{def:pi-A}. We can write:
\begin{equation}
\label{eq:cluster-size-distribution}
\pi_A(\kappa)=\p(\mathcal H_q(\kappa)|\mathcal H_q(0)). 
\end{equation}
\begin{remark} We note that the formula on the right hand side of \eqref{eq:cluster-size-distribution} can be identified precisely as the distribution of $C_r^{|k}$ considered in \cite{R13} for the cluster size distribution.
\end{remark}
\begin{proof}
By definition of $Q_q^{(\kappa)}(A)$ given in \eqref{def:Qk}, we have
$$
\mathcal H_q(\kappa)=\shift^{-q}(Q_q^{(\kappa-1)}(A))\setminus \bigcup_{i=0}^{q-1} \left(\shift^{-i}(Q_q^{(\kappa)}(A))\cap \shift^{-q}(Q_q^{(\kappa-1)}(A))\right).
$$
Observe that $ \left(\shift^{-i}(Q_q^{(\kappa)}(A))\cap \shift^{-q}(Q_q^{(\kappa-1)}(A))\right)\bigcap \left(\shift^{-j}(Q_q^{(\kappa)}(A))\cap \shift^{-q}(Q_q^{(\kappa-1)}(A))\right)=\emptyset$, for all $0\leq i\neq j\leq q-1$. To see this, assume w.l.o.g. that $0\leq i<j\leq q-1$ and take $\underline x=(x_0, x_1,\ldots)\in\mathcal R^{\N_0}$ such that $\underline x \in \shift^{-i}(Q_q^{(\kappa)}(A))\cap \shift^{-q}(Q_q^{(\kappa-1)}(A))$ then realise that $V_{i}(h(\underline x)))=1$ and $V_{\ell}(h(\underline x)))=0$ for all $\ell=i+1,\ldots,q-1$, while $i<j\leq q-1$ and $\underline x \in \shift^{-j}(Q_q^{(\kappa)}(A))\cap \shift^{-q}(Q_q^{(\kappa-1)}(A))$ means that in particular that $V_{j}(h(\underline x)))=1$, which establishes that the two events are definitely incompatible. Hence, by stationarity we have:
\begin{align*}
\p(\mathcal H_q(\kappa))&=\p(\shift^{-q}(Q_q^{(\kappa-1)}(A)))-\sum_{i=0}^{q-1} \p\left(\shift^{-i}(Q_q^{(\kappa)}(A))\cap \shift^{-q}(Q_q^{(\kappa-1)}(A))\right)\\
&=\p(Q_q^{(\kappa-1)}(A))-\sum_{i=0}^{q-1} \p\left(Q_q^{(\kappa)}(A)\cap \shift^{-q+i}(Q_q^{(\kappa-1)}(A))\right).
\end{align*}
Now, we claim that $Q_q^{(\kappa)}(A)=\bigcupdot_{i=0}^{q-1}Q_q^{(\kappa)}(A)\cap \shift^{-q+i}(Q_q^{(\kappa-1)}(A))$, where $\cupdot$ stands for disjoint union. To see this observe that 
\begin{align}
Q_q^{(\kappa)}(A)&=h^{-1}\left(\left\{\underline \omega\in\{0,1\}^{\N_0}: \underline \omega=\varpi\ldots,\;\mbox{for some $\varpi\in W_q(\kappa+1)$}\right\}\right)\nonumber\\
&=\bigcupdot_{i=0}^{q-1}h^{-1}\left(\left\{\underline \omega\in\{0,1\}^{\N_0}: \underline \omega=1\!\underbrace{0\ldots0}_{\mbox{\tiny i symbols}}\!\varpi\ldots,\;\mbox{for some $\varpi\in W_q(\kappa)$}\right\}\right)\nonumber\\
\label{eq:Q^k-recursive}
&=\bigcupdot_{j=1}^{q}Q_q^{(\kappa)}(A)\cap \shift^{-j}(Q_q^{(\kappa-1)}(A)).
\end{align}
It follows that
$$
\p(\mathcal H_q(\kappa))=\p(Q_q^{(\kappa-1)}(A))-\p(Q_q^{(\kappa)}(A)).
$$
Note that $Q_q^{(0)}(A)= \shift^{-1}(\mathscr W_{[0,q)}(A))\setminus\mathscr W_{[0,q+1)}(A)$ and $\mathcal H_q(0)=\mathscr W_{[0,q)}(A)\setminus\mathscr W_{[0,q+1)}(A)$. Therefore, by stationarity $\p(Q_q^{(0)}(A))=\p(\mathcal H_q(0))=\p(\mathscr W_{[0,q)}(A))-\p(\mathscr W_{[0,q+1)}(A))$. Recalling that $\mathcal H_q(\kappa)\subset \mathcal H_q(0)$ we obtain:
$$
\p(\mathcal H_q(\kappa)|\mathcal H_q(0))=\frac{\p(\mathcal H_q(\kappa))}{\p(\mathcal H_q(0))}=\frac{\p(Q_q^{(\kappa-1)}(A))-\p(Q_q^{(\kappa)}(A))}{\p(Q_q^{(0)}(A))}=\pi_A(\kappa).
$$
\end{proof}
\end{theorem}
\begin{remark}
Note that from \eqref{eq:Q^k-recursive}, we have $Q_q^{(\kappa)}(A)\subset\shift^{-j}(Q_q^{(\kappa-1)}(A))$, which by stationarity implies that $\p(Q_q^{(\kappa)}(A))\leq \p(Q_q^{(\kappa-1)}(A))$.
\end{remark}
From the formula \eqref{def:pi-A} we can easily derive a formula for the mean finite time cluster size distribution, which will see below to coincide with the reciprocal of the definition of the Extremal Index. 
\begin{theorem}
\label{thm:finite-mean}
Let $q\in\N$ and consider $U^{(\infty)}(A)$ defined as in \eqref{def:Uinfty} and the distribution $\pi_A$ given by \eqref{def:pi-A}. If $\p(U^{(\infty)}(A))=0$, then
$$
\sum_{j=1}^\infty j\pi_A(j)=\frac{\p(U^{(0)}(A))}{\p(Q_{q}^{(0)}(A))}.
$$
\end{theorem}
\begin{proof}
Observe that by construction, for all $n\in\N$, we have $Q_{q}^{(\kappa)}(A)\cap Q_{q}^{(j)}(A)=\emptyset$, for all $\kappa\neq j$. Moreover, $U^{(0)}(A)=\bigcup_{\kappa=0}^\infty Q_{q}^{(\kappa)}(A)\cup U^{(\infty)}(A)$ and then, by assumption, we have
$$
\p(U^{(0)}(A))=\sum_{\kappa=0}^\infty \mu_{\alpha}(Q_{q}^{(\kappa)}(A)).
$$
It follows that
\begin{align*}
\sum_{j=1}^\infty j\pi_A(j)&=\sum_{i=1}^\infty\sum_{j=i}^\infty \pi_A(j)=\sum_{i=1}^\infty\sum_{j=i}^\infty \frac{\p(Q_{q}^{(j-1)}(A))-\p(Q_{q}^{(j)}(A))}{\p(Q_{q}^{(0)}(A))}=\frac{\sum_{i=1}^\infty\p(Q_{q}^{(i-1)}(A))}{\p(Q_{q}^{(0)}(A))}\\
&=\frac{\p(U^{(0)}(A))}{\p(Q_{q}^{(0)}(A))}.
\end{align*}
\end{proof}
\begin{corollary}
If $\shift$ is ergodic w.r.t. $\p$ and $\p(\mathscr W_{[0,q+1)}(A))>0$ then $\p(U^{(\infty)}(A))=0$ and therefore the statement of the previous theorem holds.
\end{corollary}
\begin{proof}
Let $B^\infty=\bigcup_{i=0}^q\shift^{-i}(U^{(\infty)}(A))$. Observe that $\shift^{-1}(B^\infty)\subset B^\infty$ and since by invariance of $\p$ we also have $\p(\shift^{-1}(B^\infty))=\p(B^\infty)$ then $\p(\shift^{-1}(B^\infty)\triangle B^\infty)=0$, which means that by ergodicity $\p(B^\infty)=0$ or $\p(B^\infty)=1$. Since $\mathscr W_{[0,q+1)}(A)\subset (B^\infty)^c$, then the hypothesis guarantees that $\p(B^\infty)\neq1$ and the conclusion follows.
\end{proof}

\subsection{Point processes and the extremal index}
\label{subsec:point-processes}

Our goal is to keep record of the number of occurrences of $A$ on a certain time frame and then be able to provide statements regarding its asymptotic behaviour. We will do so by considering point process theory. The asymptotics comes to play by considering events that are rarer and rarer, \ie we will consider a nested sequence of sets $A_n$ such that $\lim_{n\to\infty}\p(A_n)=0$. In fact, we will use the framework of Extreme Value Theory where $\{X_0\in A_n\}$ corresponds to an exceedance $\{X_0>u_n\}$ of a threshold $u_n$, where $u_n$ is converging to the right hand point of the support of the distribution function of $X_0$ (which may be $+\infty$). We remark that there is no loss of generality in doing so because one could always find an auxiliary stochastic process $Y_0, Y_1, \ldots$ such that $\{X_0\in A_n\}=\{Y_0>u_n\}$ (see \cite{FFT10,FFT11}).

We assume that the sequence of levels $(u_n)_{n\in\N}$ satisfies the condition:
\begin{equation}
\label{un}
\lim_{n\to\infty}n\p(X_0>u_n)=\tau,
\end{equation}
for some $\tau>0$. This condition is requiring that the average frequency of exceedances of the level $u_n$ among the $n$ first observations is asymptotically constant. Note that this, in particular, implies that $\lim_{n\to\infty}u_n=\sup\{x\in\R: \p(X_0\leq x)<1\}$.

Let $E=[0,\infty)$. We say that $m$ is a \emph{point measure} on $E$ if $m=\sum_{i=1}^\infty\delta_{x_i}$, where $\delta_{x_i}$ denotes the Dirac measure supported on $x_i\in E$. We say that $m$ is \emph{simple} if all the $x_i$ are distinct and that $m$ is \emph{Radon} if $m(K)<\infty$ for all compact $K\subset E$. Consider the space $M_p(E)$ of all the Radon point measures defined on $E$ endowed with the vague topology. A \emph{point process} on $E$ is just a random element on $M_p(E)$ and we will be particularly interested on the following:
\begin{equation}
\label{def:Nn}
N_n=\sum_{i=0}^\infty \delta_{\frac i n}\I_{\{X_i>u_n\}}.
\end{equation}
Note that $N_n([0,1))$ counts the number of exceedances among the first $n$ observations of the process. Moreover, on account of \eqref{un}, for any interval $J\subset E$, we have that $\E(N_n(J))\to\tau|J|$, where $|J|$ denotes the Lebesgue measure of $J$.

Our main goal is to study the weak convergence of $N_n$. A point process $N$ on $E$ is the weak limit of $N_n$ if for any finite number of intervals of the form $J_\ell=[a_\ell,b_\ell)$, with $\ell=1,\ldots,\varsigma$, we have that the random vector $(N_n(J_1), \ldots, N_n(J_\varsigma))$ converges in distribution to $(N(J_1), \ldots, N(J_\varsigma))$ (see \cite{K86}).

We will see that under certain conditions the weak limit $N$ is a compound Poisson process, which can be described in the following way. Let $W_1, W_2,\ldots$ be an iid sequence of exponentially distributed random variables with mean $1/\eta>0$, \ie $W_i\sim \mbox{Exp}(\eta)$. Let $T_i=\sum_{j=1}^i W_i$ and $D_1, D_2,\ldots$ be an iid sequence of positive integer valued random variables independent of $T_1, T_2, \ldots$. Then $N=\sum_{i=1}^\infty D_i\delta_{T_i}$. Typically, $T_i$ corresponds to the time of appearance of the $i-th$ cluster and $D_i$ the respective size. We say that $n$ is a compound Poisson process with intensity $\eta$ and multiplicity distribution given by $\pi(\kappa)=\p(D_1=\kappa)$.

The weak convergence of $N_n$ gives a lot of information about the limiting behaviour of the order statistics of a finite sample of $X_0, X_1,\ldots$. In particular, if $M_n=\max\{X_0,\ldots,X_{n-1}\}$ we have $\{M_n\leq u_n\}=\{N_n([0,1))=0\}$. Therefore, $\lim_{n\to\infty}\p(M_n\leq u_n)=\p(N([0,1)=0)$. When we have a compound Poisson process in the limit then $\p(N([0,1)=0)=\p(W_1>1)=\e^{-\eta}$. Since in most situations $\E(N_n([0,1)))=\tau$ then $\eta=\tau/\E(D_1)$. This motivates the following definition.

\begin{definition}
\label{def:EI}
Consider a sequence $(u_n)_{n\in\N}$ such that \eqref{un} holds. We say we have an Extremal Index (EI) $0\leq\theta\leq 1$ if $\lim_{n\to\infty}\p(M_n\leq u_n)=\e^{-\theta\tau}$.
\end{definition}

Usually, we have that $\theta^{-1}=\E(D_1)$ and the EI can be interpreted as a measure of the intensity of clustering, so that $\theta=1$ means the absence of clustering. We will build examples where this relation between the EI and $\E(D_1)$ does not hold anymore.

\subsection{Convergence of point processs}
\label{subsec:convergence-PP}

In order to obtain the convergence of the point processes introduced above we will use two conditions on the dependence structure of original stochastic process $X_0, X_1, \ldots$. We introduce the notation for all $\kappa\in\N_0\cup\{\infty\}$:
$$
U^{(\kappa)}(u_n):=U^{(\kappa)}([u_n,\infty)), \quad Q_q^{(\kappa)}(u_n):=Q_q^{(\kappa)}([u_n,\infty)) \quad\mbox{and}\quad \pi_n(\kappa):=\pi_{[u_n,\infty)}(\kappa)$$
The first condition is a sort of mixing condition specially designed for this extreme analysis with applications to dynamically generated stochastic processes. It was introduced in \cite{FFT13}.
\begin{condition}[$\D_q(u_n)^*$]\label{cond:Dp*}We say that $\D_q(u_n)^*$ holds
for the sequence $X_0,X_1,\ldots$ if for any integers $t, \kappa_1,\ldots,\kappa_\zeta$, $n$ and
 any intervals of the form $I_j=[a_j,b_j)$ with $a_{j+1}\geq b_j$ for all $j=1,\ldots \zeta-1$ and such that $a_1\ge t$,
 \begin{align*}
 \Big|\p\left(Q_{q}^{(\kappa_1)}(u_n)\cap \left(\cap_{j=2}^\zeta N_n(I_j)=\kappa_j \right) \right)&-\p\left(Q_{q}^{(\kappa_1)}(u_n)\right)
  \p\left(\cap_{j=2}^\zeta N_n(I_j)=\kappa_j \right)\Big|\\ & \leq \gamma(q,n,t),
\end{align*}
where for each $n$ we have that $\gamma(q,n,t)$ is nonincreasing in $t$  and
$n\gamma(q,n,t_n)\to0$  as $n\rightarrow\infty$, for some sequence
$t_n=o(n)$.
\end{condition}

For some fixed $q\in\N_0$, consider the sequence $(t_n)_{n\in\N}$, given by condition  $\D_q(u_n)$ and let $(k_n)_{n\in\N}$ be another sequence of integers such that
\begin{equation}
\label{eq:kn-sequence}
k_n\to\infty\quad \mbox{and}\quad  k_n t_n = o(n).
\end{equation}

\begin{condition}[$\D'_q(u_n)^*$]\label{cond:D'q} We say that $\D'_q(u_n)^*$
holds for the sequence $X_0,X_1,X_2,\ldots$ if there exists a sequence $(k_n)_{n\in\N}$ satisfying \eqref{eq:kn-sequence} and such that
\begin{equation}
\label{eq:D'rho-un}
\lim_{n\rightarrow\infty}\,n\sum_{j=q+1}^{\lfloor n/k_n\rfloor-1}\p\left( Q_{q}^{(0)}(u_n)\cap \shift^{-j}(U^{(0)}(u_n))\right)
=0.
\end{equation}
\end{condition}
Note that condition $\D'_q(u_n)^*$ is just condition $D^{(q+1)}(u_n)$ in the formulation of \cite[Equation (1.2)]{CHM91}. 
\begin{remark}
\label{rem:D'-declustering-relation}
Observe that condition $\D'_q(u_n)^*$  is forbidding (or making very unlikely) the appearance of two clusters in a very short period of time (namely, within a block of size $\lfloor n/k_n\rfloor$), which means that the two declustering schemes identify clusters essentially in the same way. Moreover, in some sense, one could say that a time gap of length $q$ inhibits the possibility of the underlying periodic phenomena creating a new exceedance that should be classified as belonging to the same cluster as the preceding exceedance because the period time has been exhausted.
\end{remark}

\begin{remark}
\label{rem:roles-of-q}
Note that if condition $\D'_q(u_n)^*$ holds for some particular $q=q_0\in\N_0$, then condition $\D'_q(u_n)^*$ holds for all $q\geq q_0$. This suggests that in trying to prove the convergence of REPP, one should try the values $q=q_0$ until we find the smallest one that makes $\D'_q(u_n)^*$ hold, establishing, in this way, the run length. In the dynamical context, the following procedure has proved very useful to find the value of $q$ (see \cite{AFFR16}). Let $\underline x\in \mathcal R^{\N_0}$ and recall the definition of the first hitting time to $A\in\mathcal B$,
$
r_A(\underline x)=\min\left\{j\in\N\cup\{+\infty\}:\; \shift^j(\underline x)\in V_0^{-1}(A)\right\}.
$
The restriction of the function $r_A$ to $A$ is called the \emph{first return time function} to $A$. We define the \emph{first return time} to $A$, which we denote by $R(A)$, as the infimum of the return time function to $A$, \ie the non-negative integer
$R(A)=\inf_{\underline x\in V_0^{-1}(A)} r_A(\underline x).$
Assume that there exists $q\in\N_0$ such that
\begin{equation}
\label{eq:q-def EVL}
q:=\min\left\{j\in\N_0: \lim_{n\to\infty}R(Q_{q}^{(0)}(u_n))=\infty\right\}.
\end{equation}
Then such $q$ is the natural candidate to try to show the validity of $\D'_q(u_n)^*$.
\end{remark}

Let us define for each $n\in\N$
\begin{equation}
\label{eq:theta-n}
\theta_n=\frac{\p(Q_q^{(0)}(u_n))}{\p(U^{(0)}(u_n))},
\end{equation}
which measures the proportion of realisations of $U^{(0)}(u_n))$, \ie exceedances of $u_n$ that do not produce another exceedance in the same cluster.

If there exists $0\leq\theta\leq1$ such that $\theta=\lim_{n\to\infty}\theta_n$, then under conditions $\D_q(u_n)^*$ and $\D'_q(u_n)^*$ we have that $\lim_{n\to\infty}\p(M_n\leq u_n)=\e^{-\theta\tau}$ (see \cite{FFT12,FFT15}), which means that $\theta$ is the EI. This formula for the EI has already appeared in the work of O'Brien \cite{O87}.

In the case where the exceedance corresponds to hitting time to a cylinder set of at least length $u_n$,
this formula was also used in \cite{A06,ACG15}, with $q$ equal to the periodicity of the cylinder.

From the study developed in \cite{FFT13} and as noticed in \cite[Appendix~B]{AFV15}, we can state the following result which applies to general stationary stochastic processes. A full proof of this result can be seen in \cite{FFM18}. 
\begin{theorem}[\cite{FFT13,FFM18}]
\label{thm:convergence-REPP}
Let $X_0, X_1, \ldots$ satisfy conditions $\D_q(u_n)^*$ and $\D'_q(u_n)^*$, where $(u_n)_{n\in\N}$ is such that \eqref{un} holds. Assume that
 the limit $\theta=\lim_{n\to\infty} \theta_n$ exists, where $\theta_n$ is as in \eqref{eq:theta-n} and moreover that for each $\kappa\in\N$, the following limit also exists
\begin{equation}
\label{eq:multiplicity}
\pi(\kappa):=\lim_{n\to\infty} \pi_{n}(\kappa)=\lim_{n\to\infty}\frac{
\left(\p(Q_{q}^{(\kappa-1)}(u_n))-\p(Q_{q}^{(\kappa)}(u_n))\right)}{
\p(Q_{q}^{(0)}(u_n))}.
\end{equation}
Then the REPP\index{Rare Event Point Process - REPP} $N_n$ converges in distribution to a compound Poisson process with intensity $\theta\tau$ and multiplicity distribution $\pi$ given by \eqref{eq:multiplicity}.
\end{theorem}

Observe that by Theorem~\ref{thm:finite-mean}, for every $n\in\N$, if $\p(U^{(\infty)}(u_n))=0$, then the mean of the distribution $\pi_{[u_n,\infty)}$ is the reciprocal of $\theta_n$, \ie  $\sum_{\kappa=1}^\infty \kappa \pi_{n}(\kappa)=\theta_n^{-1}$. It follows that if there exists $0\leq\theta\leq 1$ such that $\theta=\lim_{n\to\infty} \theta_n$, then $\lim_{n\to\infty}\sum_{\kappa=1}^\infty \kappa \pi_{n}(\kappa)=\theta^{-1}$. 

We are going to build an example such that, although the latter equality holds, the same does not hold for the asymptotic distribution of the cluster size, \ie $\sum_{\kappa=1}^\infty \kappa \pi(\kappa)\neq\theta^{-1}$.

Before we do it, we revisit the examples introduced earlier in order to illustrate the usual behaviour.

\textbf{Example~\ref{exp:MMA-101} revisited}
\\
Note that by definition of the sequence $(u_n)_{n\in\N}$, we must have $\p(X_0>u_n)\to\tau\geq 0$, as $n\to\infty$. For simplicity let $\alpha_n=\p(Y_0\leq u_n)$. Then, since $\p(X_0>u_n)=2(1-\alpha_n)-(1-\alpha_n)^2$ we must have that $\alpha_n\to 1$ and $n(1-\alpha_n)\to\tau/2\geq 0$, as $n\to\infty$.

In order to exemplify the usefulness of the binary string notation introduced earlier  to describe the sets $Q_q^{(\kappa)}$, let $\ell\in\N$, $\varsigma\in\{0,1\}^\ell$ and define the cylinder $$C(\varsigma)=\{\underline\omega\in\{0,1\}^\N\colon \underline\omega=\varsigma\underline\omega^*, \mbox{ for some } \underline\omega^*\in\{0,1\}^\N\},$$
corresponding to the binary words which start with the string $\varsigma$. Then,\begin{align*}
Q_q^{(0)}(u_n)&=h^{-1}\left(C(100)\right)=\{X_0>u_n,X_1\leq u_n,X_2\leq u_n\}=\{Y_{-2}>u_n, Y_{-1},Y_0,Y_1,Y_2\leq u_n\},\\
Q_q^{(1)}(u_n)&=h^{-1}\left(C(1100)\cup C(10100)\right)=\\
&=\{Y_{0}>u_n, Y_{-1},Y_1,Y_2,Y_3, Y_4\leq u_n\}\cup \{Y_{-2}, Y_{-1}>u_n, Y_0,Y_1,Y_2,Y_3\leq u_n\};\\
Q_q^{(2)}(u_n)&=h^{-1}\left(C(11100)\cup C(110100)\cup C(101100)\cup C(1010100)\right)=\\
&=\{Y_{-1}, Y_{0}>u_n, Y_1,Y_2,Y_3,Y_4\leq u_n\}\cup \{Y_{-2},Y_{1}>u_n, Y_0,Y_2,Y_3,Y_4, Y_5\leq u_n\}\\&\qquad \cup \emptyset\cup\left( \{Y_{2}>u_n, Y_{-1},Y_{1},Y_3,Y_4,Y_5, Y_6\leq u_n\}\cap\left(\{Y_0>u_n\}\cup\{Y_{-2}>u_n\}\right)\right);
\end{align*}iid
Using that $Y_{-2},Y_{-1},Y_0,Y_1,\ldots$ is an iid sequence, we have
\begin{align*}
\p(Q_q^{(0)}(u_n))&=(1-\alpha_n)\alpha_n^4,\\
\p(Q_q^{(1)}(u_n))&=(1-\alpha_n)\alpha_n^5+(1-\alpha_n)^2\alpha_n^4;\\
\p(Q_q^{(2)}(u_n))
&=(1-\alpha_n)^2\alpha_n^4+(1-\alpha_n)^2\alpha_n^5+(1-\alpha_n)^2\alpha_n^6(1+\alpha_n).
\end{align*}
Using the formulae for the EI, \eqref{eq:theta-n}, and the multiplicity distribution, \eqref{def:pi-A}, we obtain:
\begin{align*}
\theta_n&
=\frac{(1-\alpha_n)\alpha_n^4}{2(1-\alpha_n)-(1-\alpha_n)^2}\xrightarrow[n\to\infty]{}\frac12=\theta,\\
\pi_n(1)&
=\frac{2(1-\alpha_n)^2\alpha_n^4}{(1-\alpha_n)\alpha_n^4}\xrightarrow[n\to\infty]{}0=\pi(1)
\\
\pi_n(2)&
=\frac{(1-\alpha_n)\alpha_n^5-(1-\alpha_n)^2\alpha_n^5-(1-\alpha_n)^2\alpha_n^6(1+\alpha_n)}{(1-\alpha_n)\alpha_n^4}\xrightarrow[n\to\infty]{}1=\pi(2)
\end{align*}
Observing that for $\kappa>2$, all the terms  of $\p(Q_q^{(\kappa)}(u_n))$ include a factor $(1-\alpha_n)^2$, one easily verifies that $\pi(\kappa)=0$ for all such $\kappa>2$. 

Also note that condition $\D_q(u_n)^*$ follows trivially from the fact that the process is $2$-dependent.

Regarding condition $\D'_q(u_n)$, observe that, for $j=3,4$, we have $$Q_{q}^{(0)}(u_n)\cap\{X_j>u_n\}=\{Y_{-2}, Y_j>u_n, Y_{-1},Y_0,Y_1,Y_2\leq u_n\}$$ and for $j\geq 5$, we have $$Q_{q}^{(0)}(u_n)\cap\{X_j>u_n\}=\{Y_{-2}>u_n, Y_{-1},Y_0,Y_1,Y_2\leq u_n\}\cap(\{X_{j-2}\geq u_n\}\cup\{X_{j}\geq u_n\}).$$ Then, clearly, $\p(Q_{q}^{(0)}(u_n)\cap\{X_j>u_n\})\leq2(1-\alpha_n)^2\alpha_n^4$, for all $j\geq3$.  Hence,
\[
\sum_{j=3}^{\lfloor n/k_n\rfloor} n\p(Q_{q}^{(0)}(u_n)\cap\{X_j>u_n\})\leq \frac{n}{k_n} n2(1-\alpha_n)^2\alpha_n^4 \xrightarrow[n\to\infty]{}0,
\]
by the properties of $\alpha_n$.

\textbf{Example~\ref{exp:geometric} revisited}\\
We take $u_n=-\log\tau + \log(2n)$ so that $\mu(X_0>u_n)=\tau/n$. We observe that $U^{(0)}(u_n)=(1/3-1/2n,1/3+1/2n)$ and for $\kappa\in\N_0$ we have 
$$Q_q^{(\kappa)}(u_n)=\left(\frac13-\frac1{4^\kappa}\frac\tau{2n},\frac13-\frac1{4^{\kappa+1}}\frac\tau{2n}\right)\cup\left(\frac13+\frac1{4^{\kappa+1}}\frac\tau{2n},\frac13+\frac1{4^\kappa}\frac\tau{2n}\right) \Rightarrow \mu(Q_q^{(\kappa)}(u_n))=\frac1{4^\kappa}\frac34\frac\tau n.$$
It follows that 
$$\theta_n=\frac{3\tau/4n}{\tau/n}\xrightarrow[n\to\infty]{}3/4=\theta \quad \mbox{and}\quad \pi_n(\kappa)=\frac{\frac1{4^{\kappa-1}}\frac34\frac\tau n-\frac1{4^{\kappa}}\frac34\frac\tau n}{\frac{3\tau}{4n}}\xrightarrow[n\to\infty]{}\frac1{4^{\kappa-1}}\frac34=\pi(\kappa).$$ Note that $\sum_{\kappa\geq1}\kappa\pi(\kappa)=4/3=\theta^{-1}$.

The convergence of the REPP to a compound Poisson process with a geometric multiplicity distribution is assured by the validity of conditions $\D_q(u_n)^*$ and $\D'_q(u_n)^*$, which follows from the fact that the system has decay of correlations against $L^1$. (See \cite[Chapter~4]{LFFF16} for further details).  

\textbf{Example~\ref{exp:mixture} revisited}
\\
Adjusting the computations of the previous example, we take $u_n=-\log\tau + \log(4n)$ so that $\mu(X_0>u_n)=\tau/n$. We observe that $U^{(0)}(u_n)=(1/3-1/4n,1/3+1/4n)\cup(5/7-1/4n,5/7+1/4n)$ and for $\kappa\in\N_0$ we have 
\begin{align*}Q_q^{(\kappa)}(u_n)&=\left(\frac13-\frac1{4^\kappa}\frac\tau{4n},\frac13-\frac1{4^{\kappa+1}}\frac\tau{4n}\right)\cup\left(\frac13+\frac1{4^{\kappa+1}}\frac\tau{4n},\frac13+\frac1{4^\kappa}\frac\tau{4n}\right)\cup\\
&\cup \left(\frac57-\frac1{8^\kappa}\frac\tau{4n},\frac57-\frac1{8^{\kappa+1}}\frac\tau{4n}\right)\cup\left(\frac57+\frac1{8^{\kappa+1}}\frac\tau{4n},\frac57+\frac1{8^\kappa}\frac\tau{4n}\right).
\end{align*}
Observing that $\mu(Q_q^{(\kappa)}(u_n))=\frac1{4^\kappa}\frac34\frac\tau{2n}+\frac1{8^\kappa}\frac78\frac\tau{2n}$, we have 
\begin{align*}
\theta_n&=\frac{\frac34\frac\tau{2n}+\frac78\frac\tau{2n}}{\tau/n}=\frac{13}{16}=\theta \\ \pi_n(\kappa)&=\frac{\frac1{4^{\kappa-1}}\left(\frac34\right)^2\frac\tau {2n}-\frac1{8^{\kappa-1}}\left(\frac78\right)^2\frac\tau {2n}}{\frac34\frac\tau {2n}-\frac78\frac\tau {2n}}=\frac{\frac1{4^{\kappa-1}}\left(\frac34\right)^2-\frac1{8^{\kappa-1}}\left(\frac78\right)^2}{\frac{13}8}=\pi(\kappa).\end{align*} Note that $\sum_{\kappa\geq1}\kappa\pi(\kappa)=4/3=\theta^{-1}$.

Again, the convergence of the REPP to a compound Poisson process  is assured by the validity of conditions $\D_q(u_n)^*$ and $\D'_q(u_n)^*$, which follows from the fact that the system has decay of correlations against $L^1$. (See \cite[Chapter~4]{LFFF16} for further details).

\section{Dynamical counterexamples}

Let us consider a one-dimensional family of maps with an indifferent fixed point of the \emph{Manneville-Pomeau} (MP) type.  We will be using the particular form given in \cite{LSV99}. Namely, for $\alpha>0$,
\begin{equation}
\label{def:LSV}
T=T_\alpha(x)=\begin{cases} x(1+2^\alpha x^\alpha) & \text{ for } x\in [0, 1/2)\\
2x-1 & \text{ for } x\in [1/2, 1]\end{cases}
\end{equation}
If $\alpha\in (0,1)$ then there is an absolutely continuous (w.r.t. Lebesgue) invariant probability $\mu_\alpha$, which is the case we will restrict to.
These maps have been studied in \cite{LSV99,Y99,H04} and, for each $\alpha\in (0,1)$, the system $([0,1], T_\alpha, \mu_\alpha)$ has polynomial decay of correlations.  That is, letting $\mathcal H_\beta$ denote the space of H\"older continuous functions $\phi$ with exponent $\beta$ equipped with the norm $\|\phi\|_{\mathcal H_\beta}=\|\phi\|_\infty+|\phi|_{\mathcal H_\beta}$, where $$|\phi|_{\mathcal H_\beta}=\sup_{x\neq y}\frac{|\phi(x)-\phi(y)|}{|x-y|^\beta},$$ 
 there exists $C>0$ such that for each $\phi\in \mathcal H_\beta$, $\psi\in L^\infty$ and all $t\in \N$,
\begin{equation}
\label{eq:Holder-DC}
\left| \int\phi\cdot(\psi\circ T^t)d\mu_\alpha-\int\phi d\mu_\alpha\int\psi
d\mu_\alpha\right|\leq C\|\phi\|_{\mathcal H_\beta}\|\psi\|_\infty \frac{1}{t^{\frac1\alpha-1}}.
\end{equation}
Let $h_\alpha=\frac{d\mu_\alpha}{dx}$. In \cite{H04}, Hu showed that $h_\alpha\in L^{1+\epsilon}$, with $\epsilon<1/\alpha-1$, $h_\alpha$ is Lipschitz on $[a,1]$ for all $0<a<1$ and moreover $\lim_{x\to0}\frac{h(x)}{x^{-\alpha}}=C_0>0$. Hence, for small $s>0$ we have that 
\begin{equation}
\label{eq:estimate-measure}
\mu_\alpha([0,s))\sim C_1 s^{1-\alpha},
\end{equation}
for some $C_1>0$, where the notation $A(s)\sim B(s)$ is used in the sense that  $\lim_{s\to0}\frac{A(s)}{B(s)}=1$. When the constant is unimportant, we will also use the notation $A(s)\sim_{c} B(s)$ in the sense that there is $c>0$ such that $\lim_{s\to0}\frac{A(s)}{B(s)}=c$.

Let 
$x$ be such that $T_{\alpha}(x)=y$, \ie $y=x+2^\alpha x^{1+\alpha}$. From the properties of the invariant density and \eqref{eq:estimate-measure} we get that there exists $C_1>0$ such that
\begin{equation}
\label{U_n}
\mu_\alpha([0,y))\sim {C_1} ( x^{1-\alpha}+(1-\alpha)2^{\alpha}x+o(x))
\end{equation}
\begin{equation}
\mu_\alpha([0,x))\sim {C_1} x^{1-\alpha}.
\end{equation}
\begin{equation}
\label{x_n,y_n}
\mu_\alpha([x,y))\sim_{c} (1-\alpha)2^{\alpha}x+o(x).
\end{equation}

Our goal is to study the extremal behaviour of stochastic processes arising from such dynamical systems by considering an observable function that we will denote by $\varphi:[0,1]\to \R\cup\{+\infty\}$ and defining the process $X_0, X_1,\ldots$ by
\begin{equation}
\label{eq:SP}
X_n=\varphi \circ T_\alpha^n, \quad\mbox{for all $n\in\N_0$}
\end{equation}
where $T_\alpha^n$ denotes the $n$-fold composition of $T_\alpha$ and $T_\alpha^0$ is just the identity map. The $T_\alpha$ invariance of $\mu_\alpha$ guarantees that $X_0, X_1,\ldots$ is stationary.

\subsection{A dynamical emulation of Smith's example}
\label{subsec:smith-emulation}

In this case, we are going to use the idea introduced in \cite{AFFR16} to make a balanced mixture of a behaviour associated with an EI equal to 0 with the behaviour of an EI equal to 1. For that purpose we are going to consider that the observable function $\varphi$  will be maximised at two points, namely, the point $\zeta_1=0$, which is an indifferent fixed point, and a point $\zeta_2\in[1/2, 1]$, whose orbit never hits the maximal set $\mathcal C=\{\zeta_1,\zeta_2\}$, \ie
$f_{\alpha}^j(\zeta_2)\notin\mathcal C, \; \forall j\in\mathbb N.$ One could take for example the preperiodic point $\zeta_2\in[1/2, 1]$ such that $f(\zeta_2)=\xi$, where $\xi$ is the periodic point of period 2 on $[0,1/2]$. The observable function will be designed so that the chances of starting near $\zeta_1$ or $\zeta_2$ are equally weighed. Note that if $\mathcal C=\{\zeta_1\}$, by \cite[Theorem~2]{FFTV16}, we would have an EI equal to 0, while, by \cite[Theorem~1]{FFTV16}, if $\mathcal C=\{\zeta_2\}$, the EI would be equal to 1. In this case, we will obtain an EI equal to $1/2$ which is the mean of the two possible values. 

We take the following observable:
\begin{equation}
\label{def:observable1}
\varphi(x)=g(C_1 \dist(x,\zeta_1)^{1-\alpha})\I_{[0,\delta)}+g(2h_\alpha(\zeta_2)\dist(x,\zeta_2))\I_{(\zeta_2-\delta,\zeta_2+\delta)},
\end{equation}
for some $\delta>0$, where $\dist$ denotes any given metric on $[0,1]$ and the function $g:[0,+\infty)\rightarrow {\R\cup\{+\infty\}}$ is such that $0$ is a global maximum ($g(0)$ may be $+\infty$); $g$ is a strictly decreasing bijection $g:V \to W$
in a neighbourhood $V$ of $0$; and has one of the three types of behaviour described for example in \cite[Section~4.2.1]{LFFF16}, which are quite general and essential guarantee that we do not fall into a case of degeneracy of the limiting law for the partial maxima of the stochastic process $X_0,X_1,\ldots$.

We claim that with this particular choice of type of observable $\varphi$ then the process $X_0, X_1,\ldots$ has an EI that does not coincide with the reciprocal of the mean cluster size distribution of the limiting process of $N_n$ given in \eqref{def:Nn}.
\begin{theorem}
Consider a a map $T_\alpha$ defined in \eqref{def:LSV} for some $0<\alpha<\sqrt5-2$. Let $\varphi$ be as in \eqref{def:observable1} and consider the stochastic process $X_0,X_1,\ldots$ defined by \eqref{eq:SP}. This process admits an EI $\theta=\frac12$. Moreover, the point process $N_n$ defined by \eqref{def:Nn} for such stochastic process and for a sequence of levels $(u_n)_{n\in\N}$ satisfying \eqref{un} converges in distribution to a Poisson process $N$ defined on the positive real line with intensity $\theta\tau$.
\end{theorem}

\begin{remark}
Observe that the EI obtained $\theta=1/2$ does not coincide with the reciprocal of the mean of the cluster size of the limiting process $N$, which in this case is $1$ because it turns out that $N$ is actually a Poisson process. 
\end{remark}

\begin{remark}
Nevertheless, recall that by Theorem~\ref{thm:finite-mean} we still have that $\theta=1/2$ is the reciprocal of the limit of the mean cluster size of the finite time point process $N_n$, \ie $$\theta^{-1}=\lim_{n\to\infty}\sum_{\kappa=1}\kappa\pi_n(\kappa).$$
\end{remark}

In order to prove this theorem, we apply Theorem~\ref{thm:convergence-REPP}. To that end we need to check conditions $\D_q(u_n)^*$ and $\D'_q(u_n)^*$ which we leave for Sections~\ref{subsec:D} and \ref{subsec:D'}, respectively. We are left to prove that $\theta_n$ given in \eqref{eq:theta-n} converges to $\theta=1/2$ and the finite time cluster size distribution $\pi_n$ given by $\pi_n(\kappa)=\pi_{[u_n,\infty)}(\kappa)$ as in \eqref{def:pi-A} converges to a degenerate distribution $\pi$ such that $\pi(1)=0$ and $\pi(\kappa)=0$ for all $\kappa>1$.

Letting $B_\delta(\zeta_2)=(\zeta_2-\delta,\zeta_2+\delta)$, we note that 
\begin{align*}
&\{\varphi(x)>u\}=\left(\{x:C_1|x|^{1-\alpha}<g^{-1}(u)\}\cap[0,\delta)\right)\cup\left(\{x:2h_\alpha(\zeta_2)|x-\zeta_2|<g^{-1}(u)\}\cap B_\delta(\zeta_2)\right)\\
\\
&=\left(\left\{|x-\zeta_1|<\left(\frac{1}{C_1}g^{-1}(u)\right)^{\frac1{1-\alpha}}\right\}\cap[0,\delta)\right)\cup\left(\left\{|x-\zeta_2|<\frac{1}{2h_\alpha(\zeta_2)}g^{-1}(u)\right\}\cap B_\delta(\zeta_2)\right).
\end{align*}

Defining now $y_n=\left(\frac{1}{C_1}g^{-1}(u_n)\right)^{1/(1-\alpha)}$ and $\delta_n=\frac{1}{2h_\alpha(\zeta_2)}g^{-1}(u_n)$, we obtain
\begin{align*}
U_n:=U^{(0)}(u_n)=\{\varphi(x)>u_n\}=[0,y_n)\cup B_{\delta_n}(\zeta_2)
\end{align*}
and by \eqref{U_n} we have
\begin{align*}
\mu_{\alpha}(U_n)&
=\mu_{\alpha}  ([0,y_n)) + \mu_{\alpha} ([\zeta_2-\delta_n,\zeta_2+\delta_n])\\
&\sim C_1 \left(\left(\frac{1}{C_1}g^{-1}(u_n)\right)^{1/(1-\alpha)}\right)^{1-\alpha}+g^{-1}(u_n)\\
& \sim 2 g^{-1}(u_n)
\end{align*}

Let $\tau$ be such that 
\begin{equation}
\label{tau-n}
2g^{-1}(u_n(\tau))=\frac{\tau}{n}\quad\mbox{or equivalently}\quad u_n(\tau)=g\left(\frac{\tau}{2n}\right).
\end{equation}
%

In this case,
\[
Q_{p}^{(0)}(u_n)=[x_n,y_n)\cup [\zeta_2-\delta_n,\zeta_2+\delta_n].
\]

So, by \eqref{x_n,y_n} end \eqref{tau-n}, we obtain
\begin{equation}
\label{Qp0}
\mu_{\alpha}(Q_{p}^{(0)}(u_n))\sim c (1-\alpha)2^{\alpha}x_n+o(x_n)+\frac{\tau}{2n}.
\end{equation}

Since $\mu_{\alpha}  ([0,y_n)) \sim g^{-1}(u_n)=\frac{\tau}{2n}$, then, by \eqref{U_n}, we have that
$${C_1} ( x_n^{1-\alpha}+(1-\alpha)2^{\alpha}x_n+o(x_n))\sim \frac{\tau}{2n},$$
which implies that
\begin{equation}
\label{x_n}
x_n=O\left(\left(\frac{\tau}{2n}\right)^{1/(1-\alpha)}\right).
\end{equation}

Then, by \eqref{Qp0},
\[
\mu_{\alpha}(Q_{p}^{(0)}(u_n))=O\left(\left(\frac{\tau}{2n}\right)^{1/(1-\alpha)}\right)+\frac{\tau}{2n}.
\]
In this way we easily obtain
\[
\theta=\lim_{n\rightarrow +\infty}\frac{\frac{\tau}{2n}+O\left(\left(\frac{\tau}{2n}\right)^{1/(1-\alpha)}\right)}{\frac{\tau}{n}}=\frac{1}{2}.
\]

Recall that 
\begin{equation*}
\pi_n(k)=\frac{\mu_{\alpha}\left(Q_{p}^{(k-1)}(u_n)\right)-\mu_{\alpha}\left(Q_{p}^{(k)}(u_n)\right)}{\mu_{\alpha}\left(Q_{p}^{(0)}(u_n)\right)}.
\end{equation*}

In this case,
$
Q_{p}^{(1)}(u_n)=[x_n^{(1)},x_n),
$
where
$T_{\alpha}\left(x_n^{(1)}\right)=x_n$, \ie $x_n^{(1)}+2^{\alpha}\left(x_n^{(1)}\right)^{1+\alpha}=x_n$, which implies that $x_n^{(1)}=O(x_n)$.

Consequently, by \eqref{x_n} 
\[
\pi_n(1)=\frac{\mu_{\alpha}\left(Q_{p}^{(0)}(u_n)\right)-\mu_{\alpha}\left(Q_{p}^{(1)}(u_n)\right)}{\mu_{\alpha}\left(Q_{p}^{(0)}(u_n)\right)}=\frac{O\left(\left(\frac{\tau}{2n}\right)^{1/(1-\alpha)}\right)+\frac{\tau}{2n}-O\left(\left(\frac{\tau}{2n}\right)^{1/(1-\alpha)}\right)}{O\left(\left(\frac{\tau}{2n}\right)^{1/(1-\alpha)}\right)+\frac{\tau}{2n}},
\]
which goes to 1 as $n$ goes to $\infty$ and, therefore, we must have $\pi(1)=\lim_{n\to\infty}\pi_n(1)=1$ and $\pi(\kappa)=0$ for all $\kappa> 1$. In any case, we can also easily check that
\[
\pi_n(k)=\frac{\mu_{\alpha}\left(Q_{p}^{(k-1)}(u_n)\right)-\mu_{\alpha}\left(Q_{p}^{(k)}(u_n)\right)}{\mu_{\alpha}\left(Q_{p}^{(0)}(u_n)\right)}=\frac{O\left(\left(\frac{\tau}{2n}\right)^{1/(1-\alpha)}\right)}{O\left(\left(\frac{\tau}{2n}\right)^{1/(1-\alpha)}\right)+\frac{\tau}{2n}},
\]
which goes to 0 as $n$ goes to $\infty$.


\subsection{Dynamical counterexample with periodic behaviour}
\label{subsec:another-example}

As in the previous example we use a maximal set $\mathcal C=\{\zeta_1,\zeta_2\}$ consisting of two points, where $\zeta_1=0$ is again the indifferent fixed point while  $\zeta_2\in[1/2,1]$ is a periodic point, namely, for some $p\in\mathbb N$, we have $T_{\alpha}^p(\zeta_2)=\zeta_2\mbox{ and } T_{\alpha}^j(\zeta_2)\notin\{\zeta_1,\zeta_2\}, \; \forall j\in\{1, \ldots,p-1\}.$ As proved in \cite{FFT12}, if $\mathcal C=\{\zeta_2\}$, then we would have an EI $\theta=1-\gamma^{-1}$, where $\gamma=DT_\alpha^p(\zeta_2)$ is the derivative of $T^p_\alpha$ at $\zeta_2$. Hence, in this case, we are mixing an evenly weighed EI equal to $0$ with an EI equal to $1-\gamma^{-1}$. As we will prove, the EI in this counterexample will be again the average of the two, \ie $\theta=\frac12(1-\gamma^{-1})$, which will not coincide with the reciprocal of the mean cluster size of the limiting process.

We take, as in the previous example, the following observable:
\begin{equation}
\label{def:observable2}
\varphi(x)=g(C_1 \dist(x,\zeta_1)^{1-\alpha})\I_{[0,\delta)}+g(2h_\alpha(\zeta_2)\dist(x,\zeta_2))\I_{(\zeta_2-\delta,\zeta_2+\delta)},
\end{equation}
for some $\delta>0$ and $g$ as described above. In this case we also have a counterexample where the EI cannot be identified as the reciprocal of the mean limiting cluster size distribution.
\begin{theorem}
Consider a a map $T_\alpha$ defined in \eqref{def:LSV} for some $0<\alpha<\sqrt5-2$. Let $\varphi$ be as in \eqref{def:observable2} and consider the stochastic process $X_0,X_1,\ldots$ defined by \eqref{eq:SP}. This process admits an EI $\theta=\frac12(1-\gamma^{-1})$, where $\gamma=DT_\alpha^p(\zeta_2)$. Moreover, the point process $N_n$ defined by \eqref{def:Nn} for such stochastic process and for a sequence of levels $(u_n)_{n\in\N}$ satisfying \eqref{un} converges in distribution to a compound Poisson process $N$ defined on the positive real line with intensity $\theta\tau$ and multiplicity distribution given by
\begin{equation}
\label{eq:multiplicity-periodic-mix}
\pi(\kappa)=\gamma^{-(\kappa-1)}(1-\gamma^{-1}), \qquad\mbox{for all $\kappa\in\N$}.
\end{equation}
\end{theorem}

\begin{remark}
Observe that the EI obtained $\theta=\frac12(1-\gamma^{-1})$ does not coincide with the reciprocal of the mean of the cluster size of the limiting process $N$, which in this case is 
$$\sum_{\kappa=1}\kappa\pi(\kappa)=\sum_{\kappa=1}^\infty \kappa \gamma^{-(\kappa-1)}(1-\gamma^{-1})=\frac1{1-\gamma^{-1}}.$$\end{remark}

\begin{remark}
As in the previous example, recall that by Theorem~\ref{thm:finite-mean} we still have that $\theta=\frac12(1-\gamma^{-1})$ is the reciprocal of the limit of the mean cluster size of the finite time point process $N_n$, \ie 
$$\theta^{-1}=\lim_{n\to\infty}\sum_{\kappa=1}\kappa\pi_n(\kappa).$$
\end{remark}

Again, in order to prove this theorem, we apply Theorem~\ref{thm:convergence-REPP}. To that end we need to check conditions $\D_q(u_n)^*$ and $\D'_q(u_n)^*$ which we leave for Sections~\ref{subsec:D} and \ref{subsec:D'}, respectively. We are left to prove that $\theta_n$ given in \eqref{eq:theta-n} converges to $\theta=\frac12(1-\gamma^{-1})$ and the finite time cluster size distribution $\pi_n$ given by $\pi_n(\kappa)=\pi_{[u_n,\infty)}(\kappa)$ as in \eqref{def:pi-A} converges to $\pi$ given in \eqref{eq:multiplicity-periodic-mix}.

As in the previous example, 
defining $y_n=\left(\frac{1}{C_1}g^{-1}(u_n)\right)^{1/(1-\alpha)}$ and $\delta_n=\frac{1}{2h_\alpha(\zeta_2)}g^{-1}(u_n)$, we obtain
\begin{align*}
U_n:=U^{(0)}(u_n)=\{\varphi(x)>u_n\}=[0,y_n)\cup (\zeta_2-\delta_n,\zeta_2+\delta_n)
\quad\mbox{and}\quad
\mu_{\alpha}(U_n)
\sim 2 g^{-1}(u_n).
\end{align*}

Again, we let $\tau$ to be as in \eqref{tau-n}.
In this case,
\[
Q_{p}^{(0)}(u_n)=[x_n,y_n)\cup (B_{\delta_n}(\zeta_2)\setminus T_\alpha^{-p}(B_{\delta_n}(\zeta_2))),
\]
where as before $B_{\delta_n}(\zeta_2)=(\zeta_2-\delta_n,\zeta_2+\delta_n)$.

So, by \eqref{x_n,y_n} and \eqref{tau-n}, we obtain
\begin{equation}
\label{Qp0_}
\mu_{\alpha}(Q_{p}^{(0)}(u_n))\sim c (1-\alpha)2^{\alpha}x_n+o(x_n)+\frac{\tau}{2n}(1-\gamma^{-1}).
\end{equation}

As we have seen in the previous example, $\mu_{\alpha}  ([0,y_n)) \sim g^{-1}(u_n)=\frac{\tau}{2n}$, and then, by \eqref{U_n}, we have that
$${C_1} ( x_n^{1-\alpha}+(1-\alpha)2^{\alpha}x_n+o(x_n))\sim \frac{\tau}{2n}.$$ 
Hence,
\begin{equation}
\label{x_n_}
x_n=O\left(\left(\frac{\tau}{2n}\right)^{1/(1-\alpha)}\right).
\end{equation}
Then, by \eqref{Qp0_},
\begin{equation}
\label{Qp0e}
\mu_{\alpha}(Q_{p}^{(0)}(u_n))=O\left(\left(\frac{\tau}{2n}\right)^{1/(1-\alpha)}\right)+\frac{\tau}{2n}(1-\gamma^{-1}).
\end{equation}
Gathering this information, we obtain
\[
\theta=\lim_{n\rightarrow +\infty}\frac{\frac{\tau}{2n}(1-\gamma^{-1})+O\left(\left(\frac{\tau}{2n}\right)^{1/(1-\alpha)}\right)}{\frac{\tau}{n}}=\frac{1}{2}(1-\gamma^{-1}).
\]
%
%
We compute now the multiplicity distribution. Observe that
\[
Q_{p}^{(1)}(u_n)=[x_n^{(1)},x_n)\cup (B_{\delta_n}(\zeta_2)\cap T_\alpha^{-p}(B_{\delta_n}(\zeta_2))\setminus T_\alpha^{-2p}(B_{\delta_n}(\zeta_2))),
\]
where
$T_{\alpha}\left(x_n^{(1)}\right)=x_n$, that is, $x_n^{(1)}+2^{\alpha}\left(x_n^{(1)}\right)^{1+\alpha}=x_n$, which implies that $x_n^{(1)}=O(x_n)$. Hence,
\begin{equation}
\label{pi_n2}
\mu_{\alpha}\left(Q_{p}^{(1)}(u_n)\right)=O\left(\left(\frac{\tau}{2n}\right)^{1/(1-\alpha)}\right)+\frac{\tau}{2n}\gamma^{-1}(1-\gamma^{-1})
\end{equation}

Consequently, by \eqref{Qp0e} and \eqref{pi_n2}, we have
\begin{align*}
\pi(1)&=\lim_{n\to\infty}\pi_n(1)=\lim_{n\to\infty}\frac{\mu_{\alpha}\left(Q_{p}^{(0)}(u_n)\right)-\mu_{\alpha}\left(Q_{p}^{(1)}(u_n)\right)}{\mu_{\alpha}\left(Q_{p}^{(0)}(u_n)\right)}\\
&=\lim_{n\to\infty}\frac{\frac{\tau}{2n}(1-\gamma^{-1})-\frac{\tau}{2n}\gamma^{-1}(1-\gamma^{-1})+O\left(\left(\frac{\tau}{2n}\right)^{1/(1-\alpha)}\right)}{\frac{\tau}{2n}(1-\gamma^{-1})+O\left(\left(\frac{\tau}{2n}\right)^{1/(1-\alpha)}\right)}\\
&=1-\gamma^{-1}
\end{align*}
In order to compute $\pi(2)$ we need to estimate $\mu_{\alpha}\left(Q_{p}^{(2)}(u_n)\right)$, which we do by noting
\[
Q_{p}^{(2)}(u_n)=[x_n^{(2)},x_n^{(1)})\cup (B_{\delta_n}(\zeta_2)\cap T_\alpha^{-2p}(B_{\delta_n}(\zeta_2))\setminus T_\alpha^{-3p}(B_{\delta_n}(\zeta_2))),
\]
where
$T_{\alpha}\left(x_n^{(2)}\right)=x_n^{(1)}$, \ie $x_n^{(2)}+2^{\alpha}\left(x_n^{(2)}\right)^{1+\alpha}=x_n^{(1)}$, which implies that $x_n^{(2)}=O(x_n^{(1)})=O(x_n)$. Thus,
\begin{equation}
\label{pi_n_2}
\mu_{\alpha}\left(Q_{p}^{(2)}(u_n)\right)=O\left(\left(\frac{\tau}{2n}\right)^{1/(1-\alpha)}\right)+\frac{\tau}{2n}\gamma^{-2}(1-\gamma^{-1})
\end{equation}

Consequently, by \eqref{Qp0e}, \eqref{pi_n2} and \eqref{pi_n_2}, 
\begin{align*}
\pi(2)&=\lim_{n\to\infty}\pi_n(2)=
\lim_{n\to\infty}\frac{\frac{\tau}{2n}\gamma^{-1}(1-\gamma^{-1})-\frac{\tau}{2n}\gamma^{-2}(1-\gamma^{-1})+O\left(\left(\frac{\tau}{2n}\right)^{1/(1-\alpha)}\right)}{\frac{\tau}{2n}(1-\gamma^{-1})+O\left(\left(\frac{\tau}{2n}\right)^{1/(1-\alpha)}\right)}\\
&=\gamma^{-1}-\gamma^{-2}=\gamma^{-1}(1-\gamma^{-1}).
\end{align*}
A simple inductive argument then leads to
\begin{align*}
\pi(\kappa)&=\lim_{n\to\infty}\pi_n(\kappa)\pi_n(\kappa)=\frac{\mu_{\alpha}\left(Q_{p}^{(\kappa-1)}(u_n)\right)-\mu_{\alpha}\left(Q_{p}^{(\kappa)}(u_n)\right)}{\mu_{\alpha}\left(Q_{p}^{(0)}(u_n)\right)}\\
&=\lim_{n\to\infty}\frac{\frac{\tau}{2n}\gamma^{-(\kappa-1)}(1-\gamma^{-1})-\frac{\tau}{2n}\gamma^{-\kappa}(1-\gamma^{-1})+O\left(\left(\frac{\tau}{2n}\right)^{1/(1-\alpha)}\right)}{\frac{\tau}{2n}(1-\gamma^{-1})+O\left(\left(\frac{\tau}{2n}\right)^{1/(1-\alpha)}\right)}\\
&=\gamma^{-(\kappa-1)}-\gamma^{-\kappa}=\gamma^{\kappa-1}(1-\gamma^{-1}).
\end{align*}
for all $\kappa \in\N$.


\subsection{The condition $\D_q(U_n)^*$}
\label{subsec:D}
Condition $\D_q(U_n)^*$ has been designed to be easily verified for systems with sufficiently fast decay of correlations. The argument used in \cite[Section~4.2.2]{FFTV16} allows to show that $\D_q(U_n)^*$ follows from the decay of correlations stated in \eqref{eq:Holder-DC}, as long as $\alpha<\sqrt5-2$.

\subsection{The condition $\D'_q(U_n)^*$}
\label{subsec:D'}

This subsection is dedicated to the verification of condition $\D'_q(U_n)^*$. We need to check \eqref{cond:D'q}. We will split the argument into two parts. In the first part we consider the points from $Q^{(0)}(u_n)$ that belong to a neighbourhood of $\zeta_2$ and in the second part the points from $Q^{(0)}(u_n)$ that belong to a neighbourhood of $\zeta_1$. Let $A_n=Q^{(0)}(u_n)\cap [0,1/2]$ and $B_n=Q^{(0)}(u_n)\cap [1/2,1].$ 
\subsubsection{Starting in a neighbourhood of $\zeta_2$}
We begin with the points that start in $B_n$. It is well known that the map $T_\alpha$ admits a first return time map $F_\alpha:[1/2,1]\to[1/2,1]$ given by $F_\alpha(x)=T_\alpha^{r_B(x)}(x)$, where $B=[1/2,1]$ and $r_B:B\to\N$ is the first return time to $B$, \ie $r_B(x)=\inf\{j\in\N: f_\alpha^j(x)\in B\}$. The map $F_\alpha$ has $\bar\mu_\alpha=\mu_\alpha|B$ as an invariant probability measure, is piecewise expanding and in particular qualifies as Rychlik map. Therefore, $F_\alpha$ has a strong form of decay of correlations, namely, there exist $C,a>0$ such that for all bounded variation functions $\phi$ against all $L^1$ functions $\psi$ we have
\begin{equation}
\label{eq:BV-L1-DC}
\left| \int\phi\cdot(\psi\circ F_\alpha^t)d\bar\mu_\alpha-\int\phi d\bar\mu_\alpha\int\psi
d\bar\mu_\alpha\right|\leq C\|\phi\|_{BV}\|\psi\|_1 \e^{-at}.
\end{equation}
Let $D_n=U^{(0)}(u_n)\cap B$, $E_n=U^{(0)}(u_n)\setminus D_n$ and $\tilde E_n=T_\alpha^{-1}(E_n)\cap B$. We observe that if $x\in B_n \cap T_\alpha^{-j}D_n$ then there exists $i\leq j$ such that $x\in  B_n \cap F_\alpha^{-i}D_n$ and, moreover, if $x\in B_n \cap T_\alpha^{-j}E_n$ then there exists $i\leq j-1$ such that $x\in  B_n \cap F_\alpha^{-i}\tilde E_n$. Therefore, 
$$
n\sum_{j=q+1}^{\lfloor n/k_n\rfloor-1}\bar\mu_\alpha\left( B_n\cap T_\alpha^{-j}(U^{(0)}(u_n))\right)\leq n\sum_{j=q+1}^{\lfloor n/k_n\rfloor-1}\bar\mu_\alpha\left( B_n\cap F_\alpha^{-j}(D_n\cup\tilde E_n)\right).
$$
We will use decay of correlations against $L^1$ of the first return time induced map $F_\alpha$ to estimate the last quantity on the right. Let $R_n=\inf\{j\in\N:B_n\cap F_\alpha^{-j}(D_n\cup\tilde E_n)\neq\emptyset\}$. In both examples described in sections \ref{subsec:smith-emulation} and \ref{subsec:another-example}, we have that $R_n\to\infty$ as $n\to \infty$. In the first situation, this follows since $B_n=D_n$ get arbitrarily small and close to $\zeta_2$, while $E_n$ gets arbitrarily small and close to $\zeta_1$ and the orbit of $\zeta_2$ does not hit $\mathcal C=\{\zeta_1,\zeta_2\}$. In the second situation, it follows because $B_n$, $D_n$ get arbitrarily small and close to $\zeta_2$, while $E_n$ gets arbitrarily small and close to $\zeta_1$ and since $\zeta_2$ is a repelling periodic point, by construction of $B_n$ its points take an arbitrarily increasing amount of time before having the opportunity to return to $D_n$. Using this observation, \eqref{eq:BV-L1-DC}, with $\phi=\I_{B_n}$ and $\psi=\I_{D_n\cap\tilde E_n}$, the facts that $\bar\mu_\alpha(D_n\cup\tilde E_n)=O\left(n^{-\frac1{1-\alpha}}\right)+O\left(n^{-1}\right)=O\left(n^{-1}\right)$ and $\|\phi\|_{BV}\leq 6$, we have
\begin{align*}
n\sum_{j=q+1}^{\lfloor n/k_n\rfloor-1}&\bar\mu_\alpha\left( B_n\cap F_\alpha^{-j}(D_n\cup\tilde E_n)\right)=n\sum_{j=R_n}^{\lfloor n/k_n\rfloor-1}\bar\mu_\alpha\left( B_n\cap F_\alpha^{-j}(D_n\cup\tilde E_n)\right)\\
&\leq n \sum_{j=R_n}^{\lfloor n/k_n\rfloor-1}\bar\mu_\alpha(B_n)\bar\mu_\alpha(D_n\cup\tilde E_n)+6Cn\bar\mu_\alpha(D_n\cup\tilde E_n)\sum_{j=R_n}^{\lfloor n/k_n\rfloor-1}\e^{-aj}\\
&\leq O\left(\frac1{k_n}\right)+O\left(\sum_{j=R_n}^{\infty}\e^{-aj}\right)\xrightarrow[n\to\infty]{}0.
\end{align*}

\subsubsection{Starting in a neighbourhood of $\zeta_1$}
Using the notation above, we start now with points in $A_n$. Observe that by definition of $A_n$ and the properties of $T_\alpha$, a point of $x\in A_n$ can only return to $U^{(0)}(u_n)$ after hitting the set $B$. Then if it hits $D_n$ it returns to $U^{(0)}(u_n)$ or if it hits $\tilde E_n$, it will return in the following iterate. Otherwise, if it hits $B\setminus(D_n\cup \tilde E_n)$, we must wait until its orbit hits $B$ again to have another chance of returning to $U^{(0)}(u_n)$. Hence, in order to check \eqref{cond:D'q}, we need to estimate 
$$
n\sum_{j=q+1}^{\lfloor n/k_n\rfloor-1}\mu_\alpha\left( A_n\cap T_\alpha^{-j}(D_n\cup\tilde E_n)\right)=n\sum_{j=R_n}^{\lfloor n/k_n\rfloor-1}\mu_\alpha\left( A_n\cap T_\alpha^{-j}(D_n\cup\tilde E_n)\right),
$$
where $R_n=\inf\{j\in\N:A_n\cap T_\alpha^{-j}(D_n\cup\tilde E_n)\neq\emptyset\}$.
Let $P_\alpha:L^1(\mbox{Leb})\to L^1(\mbox{Leb})$ denote the transfer or Perron-Frobenius operator given by duality from the equation 
$$
\int \phi\cdot\psi\circ T_\alpha\, d x=\int P_\alpha(\phi)\cdot \psi\, dx,
$$
where $\phi \in L^1(\mbox{Leb})$ and $\psi \in L^\infty(\mbox{Leb})$. Now, recalling that $h_\alpha> 0$, we have 
\begin{align*}
\mu_\alpha&\left( A_n\cap T_\alpha^{-j}(D_n\cup\tilde E_n)\right)=\int \I_{A_n}\cdot\I_{D_n\cup \tilde E_n}\circ T_\alpha^j\cdot h_\alpha\, dx =\int P^j_\alpha(\I_{A_n}h_\alpha)\cdot\I_{D_n\cup \tilde E_n}\, dx\\
&=\int \frac{P^j_\alpha(\I_{A_n}h_\alpha)}{h_\alpha}\cdot\I_{D_n\cup \tilde E_n}\cdot h_\alpha\, dx\leq \mu_\alpha(D_n\cup \tilde E_n)\sup_{x\in D_n\cup \tilde E_n}\frac{P^j_\alpha(\I_{A_n}h_\alpha)}{h_\alpha}.
\end{align*}
 Following now the same argument used in \cite[Section~4.2.1]{FFTV16} to estimate $\frac{P^j_\alpha(\I_{A_n}h_\alpha)}{h_\alpha}$, with the necessary adjustments (note that here $A_n=[x_n,y_n)$ where $x_n\sim_{c}\frac1{n^{1/(1-\alpha)}}$ while in \cite[Section~4.2.1]{FFTV16} $x_n$ was such that $x_n\sim_{c}\frac1{n}$) we obtain for some $C>0$, 
 $$
 \frac{P^j_\alpha(\I_{A_n}h_\alpha)}{h_\alpha} \leq C \frac{1}{n^{1/(1-\alpha)}}.
 $$
Therefore, recalling that $\mu_\alpha(D_n\cup\tilde E_n)=O\left(n^{-1}\right)$, we have 
$$
n\sum_{j=R_n}^{\lfloor n/k_n\rfloor-1}\mu_\alpha\left( A_n\cap T_\alpha^{-j}(D_n\cup\tilde E_n)\right)\leq n\frac n{k_n}\mu_\alpha(D_n\cup \tilde E_n)\frac{C}{n^{1/(1-\alpha)}}=O\left(\frac1{k_n n^{\alpha/(1-\alpha)}}\right)\xrightarrow[n\to\infty]{}0.
$$

\section{Escape of mass}
\label{sec:escape-of-mass}

We note that, in the counterexamples that we built, there exists an escape of mass, which is responsible for difference between the mean of the finite time cluster size distribution (associated to the point process $N_n$) and the mean of the limiting cluster size distribution (associated to the limiting process $N$). The loss of mass can be immediately detected by looking at the average number of rare events (exceedances of $u_n(\tau)$) recorded by both $N_n$ and $N$. Indeed, observe that $\E(N_n([0,1))=\tau$, where $\tau$ is given by \eqref{un}, while $\E(N([0,1))=\frac12\tau$, in the case considered in Section~\ref{subsec:smith-emulation}, and  $\E(N([0,1))=\frac12(1-\gamma^{-1})\tau$, in the case considered in Section~\ref{subsec:another-example} (recall that $\gamma>1$). This means that the limiting processes have lost half of the mass relative to extremal events detected, in the first case, and more than half, in the second case.

The main goal of this section is to try to provide an explanation for the question: how did mass disappear?

We consider two dimensional point processes as studied in \cite{FFM17a}, namely,
\begin{equation}
N_n^{(2)}=\sum_{j=0}^\infty \delta_{\left(j/n,u_n^{-1}(X_j)\right)}
\end{equation} 
We are assuming that for each $n\in\N$, the threshold function $u_n(\tau)$ is continuous and strictly decreasing in $\tau$. We can define the inverse function $u_n^{-1}$. This function can be thought of as the asymptotic frequency associated to a given threshold on the. range of the r.v. $X_0$. This point process is defined on the bi-dimensional space $E^2=[0,+\infty)\times [0,+\infty)$ and keeps record both of the times of occurrence of events and also of their severity, in the sense that a point with a vertical coordinate close to $0$ corresponds to a severe or abnormally high observation (whose corresponding asymptotic frequency is very low, \ie very few exceedances of the corresponding threshold are expected). 

The weak convergence of these point processes is a very powerful tool to obtain other results such as convergence of record point processes extremal processes, limiting laws for the maxima, which can all be settled very easily through the continuous mapping theorem and a suitable projection (see \cite{R87}, for example). In particular, note that if we define $H_{\tau}:E^2\to E$, by $H_\tau(t,y)=t\cdot\I_{[0,\tau)}(y)$ then $H_\tau(N_n^{(2)})=N_n$.

\subsection{The regular periodic case when the EI is the reciprocal of the mean limiting cluster size distribution}

In order to understand how the mass escapes, we are going to consider first the usual case where the EI coincides with the reciprocal of the cluster size distribution. Suppose that the observable $\varphi:[0,1]\to \R\cap\{+\infty\}$ is maximised at a single periodic point $\zeta\in[1/2,1]$, \ie 
\begin{equation}
\label{eq:observable-periodic}
\varphi(x)=g(|x-\zeta|),
\end{equation} where $g$ is as above and, for definiteness $\zeta$ is the periodic point of period 2 sitting on $[1/2,1]$. Let $\gamma=DT_\alpha(T_\alpha(\zeta))\cdot DT_\alpha(\zeta)$. 
Considering a stochastic process $X_0, X_1, \ldots$ defined as in \ref{eq:SP} for such $\varphi$ and given a sequence $(u_n(\tau))_{n\in\N}$ as in \eqref{un}, by \cite{FFT13}, we have that there exists an EI $\theta=(1-\gamma^{-1})$ and $N_n$ given in \eqref{def:Nn} converges to a compound Poisson process $N$ of intensity $\theta\tau$ with a geometric cluster size distribution, \ie $\pi(\kappa)=\p(D_i=\kappa)=\theta(1-\theta)^{\kappa-1}$. Moreover, by \cite{FFM17a}, we have that $N_n^{(2)}$ converges weakly to 
\begin{equation}
\label{eq:repelling-limit-REPP-2d}
N^{(2)}=\sum_{i,j=1}^\infty\sum_{\ell=0}^{\infty} \delta_{(T_{i,j},\,\gamma^\ell\cdot U_{i,j})},
\end{equation}
where the matrices $(T_{i,j})_{i,j\in\N}$ and $(U_{i,j})_{i,j\in\N}$ are mutually independent and obtained in the following way.
%
%
 Let $(W_{i,j})_{i,j\in\N}$ be a matrix of iid r.v. with common $\text{Exp}(\theta)$ distribution and consider $(T_{i,j})_{i,j\in\N}$ given by: 
$T_{i,j}=\sum_{\ell=1}^j W_{i,\ell}.$ 
Note that the rows of $(T_{i,j})_{i,j\in\N}$ are independent. Let $(U_{i,j})_{i,j\in\N}$ be a matrix of independent r.v. such that, for all $j\in\N$, the r.v. $U_{i,j}\stackrel[]{D}{\sim} \mathcal U_{(i-1,i]}$, \ie $U_{i,j}$ has a uniform distribution on the interval $(i-1,i]$. 

The point process $N^{(2)}$ can be described in the following way, first one obtains the points of bi-dimensional Poisson process on $E^2$ with $\theta\cdot\mbox{Leb}$ as its intensity measure, where $\theta\cdot\mbox{Leb}([a,b)\times[c,d))=\theta(b-a)(d-c)$, and then for every such point created we put a vertical pile of points above it, such that the distance to the original point follows a geometric law, namely, their second coordinate is the original one multiplied by a power of $\gamma$. The idea is that the observations within a cluster in $N_n^{(2)}$ appear closer and closer in time and as $n$ goes to $\infty$, eventually, they get to be aligned on the same vertical line for $N^{(2)}$. On the other hand, the dynamics near $\zeta$ tell us that if an orbit enters a very close neighbourhood of $\zeta$ then it gets repelled away at a rate given by $\gamma$, which explains the vertical distribution of the points. To be more precise, we note that $u_n^{-1}(z)\sim n2h_\alpha(\zeta)g^{-1}(z)$. Now, say that $X_j$ is so large that $u_n^{-1}(X_j)=1$, which means that the point $(j/n,1)$ is charged by the point process $N_n^{(2)}$. Then $|T_\alpha^j(x)-\zeta|\sim \frac1{2h_\alpha(\zeta)n}$. Since $DT^2_\alpha(\zeta)=\gamma$, then for large $n$ it follows that $|T_\alpha^{j+2}(x)-\zeta|\sim \frac{\gamma}{2h_\alpha(\zeta)n}$, $|T_\alpha^{j+4}(x)-\zeta|\sim \frac{\gamma^2}{2h_\alpha(\zeta)n}$ and so forth. Recalling that the points  $(\frac{j+2}{n},u_n^{-1}(X_{j+2})), (\frac{j+4}{n},u_n^{-1}(X_{j+4})),\ldots$ will also be charged by $N_n^{(2)}$ and since by the previous computations and the form of $\varphi$ we have $u_n^{-1}(X_{j+2})\sim \gamma, u_n^{-1}(X_{j+4})\sim\gamma^2,\ldots$, then one realises that, in the limit process $N^{(2)}$, these cluster points get vertically aligned and distributed according to the powers of $\gamma$.

Also observe that $H_\tau(N^{(2)})=N$. In fact, exceedances of the level $u_n(\tau)$ correspond to points with second coordinate less than $\tau$ and the cluster size can be easily interpreted as the number of points in each vertical pile still below the threshold $\tau$ that project on the same time event. See Figure~\ref{fig:process}.
\begin{figure}[h]
\includegraphics[height=9cm]{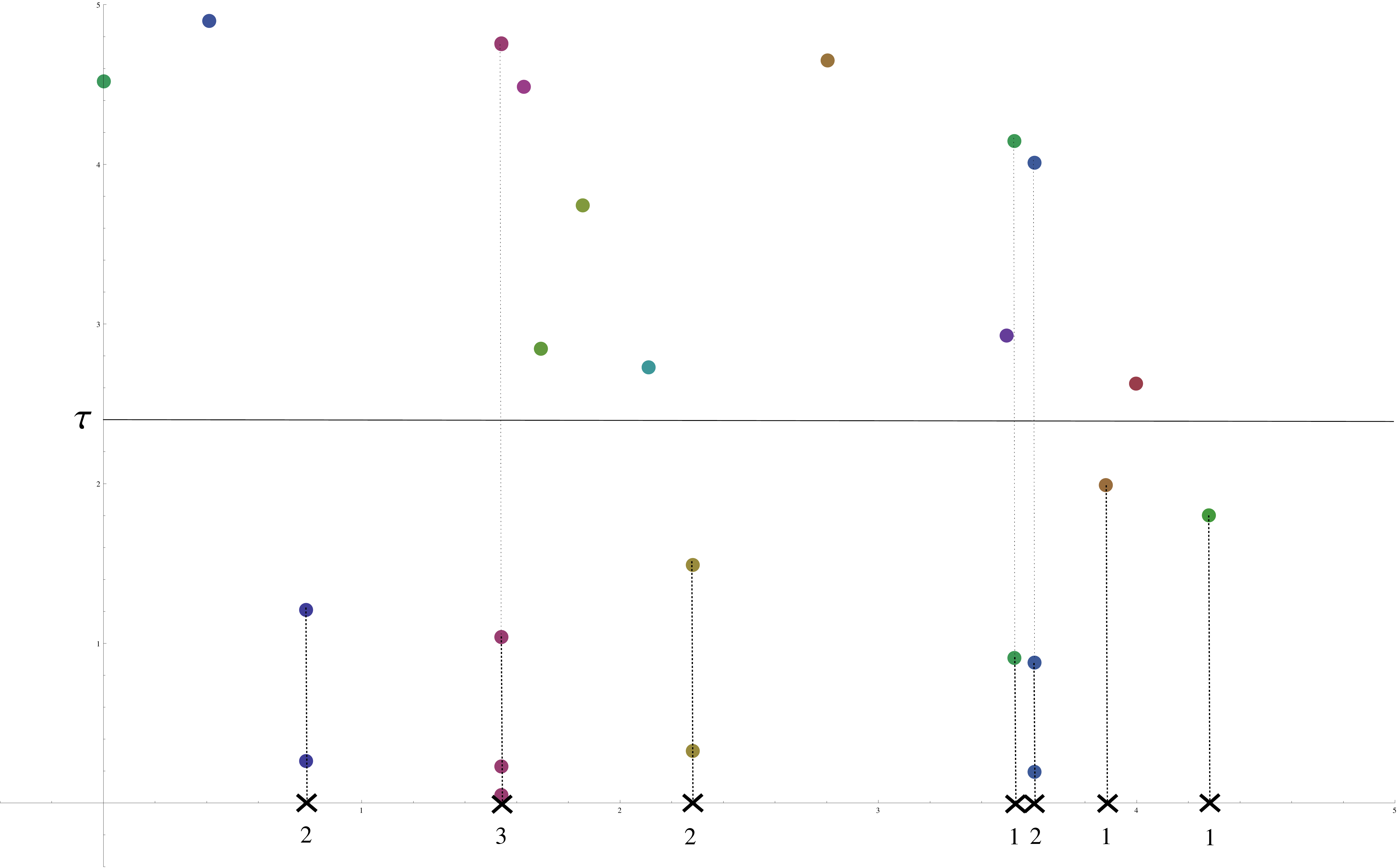}
\caption{Simulation of the bi-dimensional process $N^{(2)}$ in the case of the periodic point where the observable $\varphi$ is given by \ref{eq:observable-periodic}. The picture also shows how the projection $H_\tau$ works to obtain $N=H_\tau(N^{(2)})$, which is a compound Poisson process on the line, where the crosses represent the Poisson time events and the numbers below them the respective multiplicity (cluster size).}
\label{fig:process}
\end{figure}

\subsection{The dynamical counterexample with no periodicity mixed with the indifferent fixed point}
Assume now that the observable $\varphi$ is given as in \eqref{def:observable1}. In this case the limiting process is a bi-dimensional Poisson process with intensity measure $\tfrac12\cdot\mbox{Leb}$, which can be written as:
$$
N^{(2)}=\sum_{i,j=1}^\infty \delta_{(T_{i,j},\,U_{i,j})},
$$
where $T_{i,j}$ and $U_{i,j}$ are as above with $\theta=1/2$. Note that in this case there are no vertical piles of points as before. However, the process $N_n^{(2)}$ does have clustering points. There are two phenomena that help to explain how do they disappear.

On one hand, if we consider that $X_j$ is a very large observation that results from the orbit entering a very small vicinity of $0$ at time $j$. From \eqref{tau-n}, we have that $u_n(\tau)=g(\tau/(2n))$, which implies that $u_n^{-1}(z)=2ng^{-1}(z)$. For definiteness, let us assume that $u_n^{-1}(X_j)=1$, which means that the point $(j/n,1)$ is charged by the point process $N_n^{(2)}$. Moreover, since $\zeta_1=0$ is an indifferent fixed point then the orbit will linger around $0$ for a long time which creates clustering and the points  $\left(\frac{j+1}{n}, u_n^{-1}(X_{j+1})\right), \left(\frac{j+2}{n}, u_n^{-1}(X_{j+2})\right),\ldots$, which are also charged by $N_n^{(2)}$, will still be close to $(j/n,1)$. As in the previous example, the points on the same cluster will end up vertically aligned because of the horizontal contraction caused by the normalisation consisting on dividing by $n$.  However, in this case, something interestingly different occurs in the vertical direction. Namely, since $X_j\sim g(C_1(T_\alpha(x))^{1-\alpha})$, then $u_n^{-1}(X_j)\sim2nC_1(T_\alpha^j(x))^{1-\alpha}$, which in turn implies that $T_\alpha^j(x)\sim\left( \frac1{2nC_1}\right)^{\frac1{1-\alpha}}$. Now, observe that 
$$
u_n^{-1}(X_{j+1})\sim 2nC_1\left(\left( \frac1{2nC_1}\right)^{\frac1{1-\alpha}}+2^\alpha \left( \frac1{2nC_1}\right)^{\frac{1+\alpha}{1-\alpha}}\right)^{1-\alpha}\sim 1+(1-\alpha)2^\alpha\left( \frac1{2nC_1}\right)^{\frac{\alpha+\alpha^2}{1-\alpha}}\sim 1.
$$ 
Similarly, we obtain that $u_n^{-1}(X_{j+2})\sim1$ and so on. Therefore, do not only the points of the same cluster get vertically aligned but they also get horizontally aligned, \ie they collapse to a single point in $N^{(2)}$. So these clusters collapse to one point. On the other hand the appearance of a cluster becomes less and less frequent since the mass concentrated at each cluster (which collapses to one point in the limit) is growing and must be compensated by a smaller and smaller frequency so that the mean of the mass in the clusters observed in $N_n^{(2)}$ below the threshold $\tau$ is approximately $\tau/2$. (Recall that the remaining $\tau/2$ correspond to the mass points associated with entrances near $\zeta_2$ for which there is no clustering). In fact, the frequency of clusters of exceedances above $u_n(\tau)$ observed in $N_n^{(2)}$ is of the order of $\frac{\mu_\alpha([x_n,y_n))}{\mu_\alpha(U_n)}\frac{\tau}2$ which becomes negligible when compared to the mean frequency $\tau/2$ corresponding to the exceedances with no clustering coming from entrances near $\zeta_2$. In the limit their asymptotic time frequency is actually $0$. Hence, in $N^{(2)}$ we only observe the contribution from the entrances near $\zeta_2$. This explains the loss of half of the mass. 

\subsection{The dynamical counterexample with a periodic point mixed with the indifferent fixed point}
For the observable $\varphi$ given in \eqref{def:observable2}, the limiting bi-dimensional process process can be written as:
$$
N^{(2)}=\sum_{i,j=1}^\infty\sum_{\ell=0}^{\infty} \delta_{(T_{i,j},\,\gamma^\ell\cdot U_{i,j})},
$$
where $T_{i,j}$ and $U_{i,j}$ are as above with $\theta=\frac12(1-\gamma^{-1})$. Recall that $T_{i,j}$ are defined as sums of the waiting times $W_{i,j}$ which follow an $\text{Exp}(\theta)$ distribution. In this case, we have two types of clustering of exceedances observed in $N_n^{(2)}$, namely the ones corresponding to entrances in $U_n$ near $\zeta_1=0$ and entrances near the periodic point $\zeta_2$. As in the previous case, the first type of clusters collapse to one point and since their asymptotic frequency is $0$, the limiting process $N^{(2)}$ does not show any sign of their appearance. In fact, in $N^{(2)}$ one can only detect the presence of the second type of clusters, which are identical to the ones described in the periodic case example, except for the fact that their asymptotic frequency is half of what one would see in the periodic case.

\bibliographystyle{abbrv}

\bibliography{CounterExample}

\end{document}